\documentclass[11pt]{amsart}
\usepackage[utf8]{inputenc}

\usepackage{amsmath,amsthm}
\usepackage{amssymb}
\usepackage{epsfig,tikz}
\usetikzlibrary{cd}
\usepackage{graphicx}
\usepackage{multicol}
\usepackage{asymptote}
\usepackage[all,cmtip]{xy}
\usepackage{mathtools}
\usepackage{enumitem}
\usepackage[tmargin=1in,bmargin=1in,lmargin=1in,rmargin=1in]{geometry}

\usepackage{hyperref}
\hypersetup{
    unicode=false,          
    pdftoolbar=true,        
    pdfmenubar=true,        
    pdffitwindow=false,     
    pdfstartview={FitH},    
    pdfnewwindow=true,      
    colorlinks=true,       
    linkcolor=blue,          
    citecolor=red,        
    filecolor=magenta,      
    urlcolor=cyan           
}

\newcommand{\Mbar}{\overline{M}}

\newcommand{\x}{\underline{x}}
\newcommand{\y}{\underline{y}}
\newcommand{\uu}{\underline{u}}
\newcommand{\vv}{\underline{v}}
\newcommand{\ww}{\underline{w}}

\newcommand{\bP}{\mathbb{P}}
\newcommand{\PP}{\mathbb{P}}

\newcommand{\beq}{\begin{equation}}
\newcommand{\eeq}{\end{equation}}

\newcommand{\injto}{\hookrightarrow}

\newcommand{\defn}{\textbf}

\newcommand{\emb}{\Omega}

\renewcommand{\MR}{\mathrm{MR}}

\newcommand{\localcurve}{C''}
\newcommand{\localyvector}{y''}
\newcommand{\localcurvewithpn}{C'''}

\newtheorem{thm}{Theorem}
\newtheorem{lemma}[thm]{Lemma}
\newtheorem{prop}[thm]{Proposition}
\newtheorem{corollary}[thm]{Corollary}

\numberwithin{thm}{section}
\numberwithin{equation}{section}
\numberwithin{figure}{section}

\theoremstyle{definition}

\newtheorem{example}[thm]{Example}
\newtheorem{definition}[thm]{Definition}
\newtheorem{notation}[thm]{Notation}

\newtheorem{remark}[thm]{Remark}

\newtheorem{strategy}[thm]{Strategy}

\newtheorem*{MainThm}{Theorem \ref{thm:main}}

\title[A proof of a conjecture of Monin and Rana on $\Mbar_{0,n}$]{A proof of a conjecture by Monin and Rana on \\equations defining $\overline{M}_{0,n}$}

\author{Maria Gillespie}
\thanks{Maria Gillespie was partially supported by NSF DMS award number 2054391.}
\address{Department of Mathematics, Colorado State University, Fort Collins, CO, USA} \email{maria.gillespie@colostate.edu} 

\author{Sean T. Griffin}
\thanks{Sean T. Griffin was partially supported by NSF Grant DMS-1439786 while in residence at the Institute for Computational and Experimental Research in Mathematics in Providence, RI in Spring 2021.}
\address{Department of Mathematics, University of California Davis, Davis, CA, USA}
\email{stgriffin@ucdavis.edu}

\author{Jake Levinson}
\thanks{Jake Levinson was partially supported by NSERC Discovery Grant RGPIN-2021-04169.}
\address{Department of Mathematics, Simon Fraser University, Burnaby, BC, Canada}
\email{jake\_levinson@sfu.ca}

\date{\today}

\begin{document}

\maketitle

\begin{abstract}
    Monin and Rana conjectured a set of equations defining the image of the moduli space $\overline{M}_{0,n}$ under an embedding into $\mathbb{P}^1\times \cdots\times \mathbb{P}^{n-3}$ due to Keel and Tevelev and verified the conjecture for $n\leq 8$ using Macaulay2. We prove this conjecture for all $n$.
\end{abstract}

\section{Introduction}
Let $\Mbar_{0,n}$ denote the Deligne--Mumford moduli space of connected, $n$-marked stable curves of genus zero. Keel and Tevelev \cite{KeT}, in their work on the log canonical embedding of $\overline{M}_{0,n}$, defined an embedding $$\Omega_{n}:\Mbar_{0,n+3}\injto \mathbb{P}^1\times\PP^2\times \PP^3 \times \cdots\times\mathbb{P}^{n},$$
following Kapranov's description of $\overline{M}_{0,n}$ as a space of Veronese curves \cite{Ka1}. The composition of $\Omega_n$ with the Segre embedding gives the embedding corresponding to the log canonical divisor $K_{\overline{M}_{0,n}} + \partial \overline{M}_{0,n}$, making it and $\Omega_n$ in some sense the most natural ways to realize $\overline{M}_{0,n}$ as a projective variety. The multidegrees of $\Omega_n$ and similar embeddings have also recently been studied combinatorially~\cite{CGM, GGL-tournaments, GGL-hyperplanes,Goldner,silversmith2021crossratio}.

In \cite{MonRan}, Monin and Rana gave a set of cubic equations in the homogeneous coordinates of $\mathbb{P}^1\times\cdots\times\mathbb{P}^{n}$ that conjecturally define $\Omega_{n}(\Mbar_{0,n+3})$ as a subscheme of $\mathbb{P}^1\times\cdots\times\mathbb{P}^{n}$. The cubic equations arise as the $2\times 2$ minors of several $2\times k$ matrices, for various $k$. In this respect Monin and Rana's equations resemble the well-known equations for Veronese and Segre embeddings of projective spaces and the Pl\"ucker equations for Grassmannians. We will refer to the ideal generated by these equations as $I_{n}$, and the subscheme cut out by $I_{n}$ as $\MR_{n}$.

\begin{example}\label{ex:M05}
  For $n=5$, choosing coordinates $[u:v]\times [x:y:z]$ for $\PP^1\times \PP^2$, the ideal $I_2$ is generated by the unique $2\times 2$ minor of the matrix $$\begin{pmatrix}
  u(x-z) & v(y-z) \\
  x & y
  \end{pmatrix}.$$ Thus, $\Mbar_{0,5}$ is cut out by the degree (1,2) equation $uy(x-z)=vx(y-z)$ in $\PP^1\times \PP^2$. 
\end{example}

In this paper, we prove this conjecture (\cite[Conjecture 1.2]{MonRan}):

\begin{thm}\label{thm:main}
  The ideal $I_{n}$ cuts out $\Omega_{n}(\Mbar_{0,n+3})$ scheme-theoretically in $\PP^1\times \cdots \times \PP^{n}$.  That is, we have $\MR_{n}=\Omega_{n}(\Mbar_{0,n+3})$ as subschemes.
\end{thm}

Monin and Rana previously showed by direct calculation that, for all $n$, the image $\Omega_{n}(\Mbar_{0,n+3})$ satisfies the equations generating $I_{n}$:

\begin{prop}[{\cite[Lemma 1.1]{MonRan}}]\label{prop:inclusion}
  We have $\Omega_{n}(\Mbar_{0,n+3})\subseteq \MR_{n}$.
\end{prop}

They also proved Theorem \ref{thm:main} for $n+3\le 8$ using Macaulay2. Theorem \ref{thm:main} shows that $I_n$ defines $\Mbar_{0,n+3}$ as a subscheme, for all $n$. We note, however, that Monin and Rana have also shown that $I_n$ is in general not saturated \cite{MonRan}.

Using Proposition \ref{prop:inclusion} as a starting point, we prove the theorem in two main steps: first, we show the equality $\MR_n = \Omega_n(\overline{M}_{0,n+3})$ holds set-theoretically (Section \ref{sec:set-theoretic}). We then show it holds scheme-theoretically as well (Section \ref{sec:scheme-theoretic}). 

For the set-theoretic equality, by Proposition \ref{prop:inclusion} it suffices to show that every point $\vec{x} \in \PP^1\times \cdots \times \PP^{n}$ satisfying the equations defining $\MR_{n}$ is the image under $\Omega_{n}$ of some curve $C\in \Mbar_{0,n+3}$.  We first show this for points $\vec{x}$ having all nonzero and distinct coordinates in every $\PP^i$ factor (Proposition \ref{prop:interior}), which correspond to curves in the interior of $\Mbar_{0,n+3}$.  For points on the boundary, we use the boundary stratification via trees and an inductive combinatorial construction to build an appropriate tree for every $\vec{x} \in \MR_n$.  We then combine this construction with the interior case to construct the specific curve $C$ associated to $\vec{x}$ (Theorem \ref{thm:set-theoretic}).

For equality as schemes, the key idea is to show that the tangent spaces agree in dimension at every point. The result then follows by the set-theoretic equality and the fact that $\Mbar_{0,n+3}$ itself is smooth. We proceed by showing, by induction on $n$, that there are sufficiently many linearly independent equations cutting out the tangent space at any given point $\vec{x}\in \MR_{n}$ (Strategy \ref{strat:summary}).  Restricting attention to the tangent space makes the analysis considerably more tractable, since it reduces the cubic equations defining $\MR_{n}$ to linear equations on the tangent space.  We construct sufficiently many linearly independent tangent equations from the Monin--Rana equations via a combinatorial analysis of the branches of the dual tree of the curve corresponding to $\vec{x}$, at the vertex where the $n$th marked point is attached. As such, our approach also constructs an explicit list of $\binom{n} {2}$ equations defining $T_{\vec{x}}(\overline{M}_{0,n+3}) \subseteq T_{\vec{x}}(\mathbb{P}^1 \times \cdots \times\mathbb{P}^n)$. (See Example \ref{ex:example}.)

\subsection{Acknowledgments}
We thank Renzo Cavalieri and Rick Miranda for helpful conversations pertaining to this work.

\section{Background and Notation}

We now establish some notation, definitions, and known results that we will use throughout.

\subsection{The moduli spaces \texorpdfstring{$\Mbar_{0,n+3}$}{M0n-bar}}

The moduli space $\Mbar_{0,n}$ parameterizes genus $0$ stable curves with $n$ labeled marked points.  Each such curve consists of a finite number of copies of $\PP^1$ glued together at simple nodes, such that each $\PP^1$ component is \textbf{stable}, meaning that it has at least three total \textbf{special points} (defined as nodes or marked points).  In this paper, we draw the irreducible $\PP^1$ components as circles, as in Figure \ref{fig:branches}.   The moduli space is connected, smooth, and proper.

Since $\Mbar_{0,n}$ is only nonempty for $n\ge 3$, we shift indices and consider the moduli space $\Mbar_{0,n+3}$. We fix the totally ordered set $$S=\{a<b<c<1<2<\cdots<n\}$$ and label the marked points on any curve $C\in \Mbar_{0,n+3}$ by $a,b,c,1,2,\ldots,n$.  We therefore will refer to the moduli space as $\Mbar_{0,S}$ rather than $\Mbar_{0,n+3}$ throughout.  

\begin{remark}
Our labeling set $a, b, c, 1, \ldots, n$ is convenient because, as we explain in Subsection~\ref{sec:BoundaryStrata}, the $i$-th component of $\Omega_n$ describes coordinates obtained by looking at the curve near the marked point $p_i$, where $i$ ranges from $1$ to $n$ (i.e. $i$ is never $a, b$ or $c$). This corresponds to the fact that $\Omega_n$ is constructed using the \textit{psi classes} $\psi_1$ (on $\Mbar_{0, \{abc1\}}$), $\psi_2$ (on $\Mbar_{0,\{abc12\}}$), \ldots, $\psi_n$ (on $\Mbar_{0, S}$). These labels also make it easier to track the combinatorial arguments in Sections \ref{sec:set-theoretic} and \ref{sec:scheme-theoretic}.
\end{remark}

We write $$\pi_n:\Mbar_{0,S}\to \Mbar_{0,S\setminus n}$$ for the \textbf{forgetting map} obtained by forgetting the marked point labeled $n$ (and collapsing any resulting unstable components).

\subsection{Boundary strata and trees}\label{sec:BoundaryStrata}

 The \textbf{dual tree} of a point in $\Mbar_{0,S}$ is the leaf-labeled tree formed by drawing a vertex in the center of each $\PP^1$ circle and then connecting this vertex to each marked point on its circle and each vertex on an adjacent circle. The resulting graph is guaranteed to be a tree since the curve has genus $0$.
 
A tree is \textbf{trivalent} if every vertex has degree $1$ or $3$ and at least one vertex has degree $3$, and it is \textbf{at least trivalent} or \textbf{stable} if it has no vertices of degree $2$ and at least one vertex of degree $\ge 3$.  The dual tree of any stable genus $0$ curve is a stable tree.  

The \textbf{interior} of $\Mbar_{0,S}$ is the open set $M_{0,S}\subset \Mbar_{0,S}$ consisting of all the curves that have a single $\PP^1$ with all distinct marked points. The points of the interior correspond to those whose dual tree consists of a central node with $|S|$ leaves attached. 

The \textbf{boundary} of $\Mbar_{0,S}$ is the complement of the interior, consisting of the points corresponding to stable curves with more than one irreducible component.  If $T$ is a stable tree whose leaves are labeled by $S$, the corresponding \textbf{boundary stratum} $X_T$ is the closure of the set of all stable curves whose dual tree is $T$.

 For a stable tree $T$, let $v_i \in T$ be the internal vertex adjacent to leaf edge $i$.  We refer to the connected components of $T \setminus \{v_i\}$ (defined by vertex deletion) as the \defn{branches of $T$ at $v_i$}.  See Figure \ref{fig:branches} for an example of these concepts. 

\begin{figure}
    \centering
    
    \includegraphics{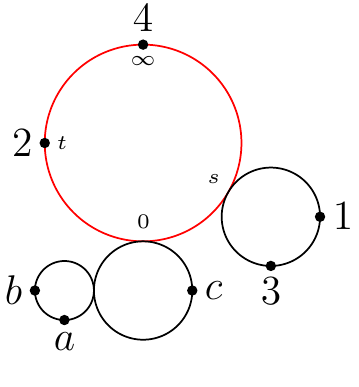}\hspace{2cm}\includegraphics{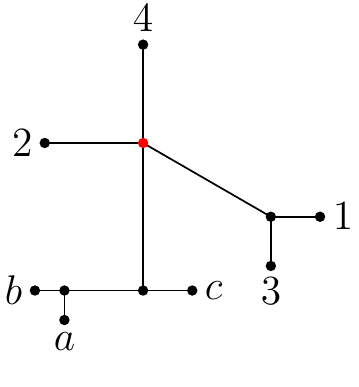}\hspace{2cm}\includegraphics{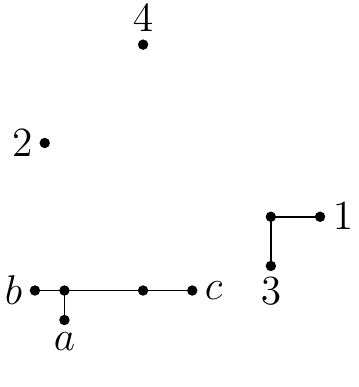}

    \caption{The stable curve at left has the dual tree shown at center, with its disconnected branches at $v_4$ shown at right. The labels $a,b,c$ are on the same branch at $v_4$, and $1,3$ are on the same branch.  The label $2$ is on its own branch, as is $4$.}
    \label{fig:branches}
\end{figure}

\subsection{The Kapranov map and the embedding \texorpdfstring{$\Omega_n$}{Omega-n}}\label{sec:Kapranov}

For all facts stated throughout this subsection, we refer the reader to Kapranov's paper \cite{Ka1}, in which the Kapranov morphism below was originally defined.  

The \textbf{$n$th cotangent line bundle} $\mathbb{L}_n$ on $\Mbar_{0,S}$ is the line bundle whose fiber over a curve $C\in \Mbar_{0,S}$ is the cotangent space of $C$ at the marked point $n$. The corresponding \textit{$\psi$ class} is $\psi_n = c_1(\mathbb{L}_n)$, the first Chern class of this line bundle. The corresponding map to projective space, which by abuse of notation we call \[\psi_n : \Mbar_{0,S} \to \mathbb{P}^{|S|-3},\]
is called the \textbf{Kapranov morphism}. We coordinatize the Kapranov map as follows.  Given a curve $C$ in the interior $M_{0,S}$, write $p_a,p_b,p_c,p_1,\dots,p_n$ for the coordinates of the $n+3$ marked points on the unique component of $C$, after choosing an isomorphism $C\cong \bP^1$. With these coordinates, the restriction of $\psi_n$ to the interior $M_{0,S}$ is given by
\begin{equation} \label{eq:kap-interior}
    \psi_n(C) = \bigg[
    \frac{p_a - p_b}{p_n - p_b} : 
    \frac{p_a - p_c}{p_n - p_c} : 
    \frac{p_a - p_1}{p_n - p_1} : 
    \frac{p_a - p_2}{p_n - p_2} :
    \cdots :
    \frac{p_a - p_{n-1}}{p_n - p_{n-1}}
    \bigg].
\end{equation}

The coordinates $p_i$ are only well-defined up to M\"obius transformations, and it is convenient to choose coordinates on $C$ in which $p_a = 0$ and $p_n = \infty$. The map then simplifies to
\begin{equation} \label{eq:kap-interior-2}
    \psi_n(C) = [p_b : p_c : p_1 : p_2 : \cdots : p_{n-1}].
\end{equation}

We now describe how to use the above formulas to compute $\psi_n$ on boundary strata, i.e. reducible stable curves $C$. Essentially, $\psi_n$ reduces to a smaller Kapranov morphism using the irreducible component of $C$ containing $p_n$ (followed by a linear map into $\mathbb{P}^n$).

\begin{prop}\label{prop:PsiCoordinates}
  Suppose $C\in \Mbar_{0,S}$ has dual tree $T$, and $\psi_n(C)=[t_b:t_c:t_1\cdots:t_{n-1}]$.  Then:
  \begin{enumerate}
      \item We have $t_j=t_k$ if and only if the marked points $\{j, k\}$ are on the same branch of $T$ at $v_n$.
      \item We have $t_j=0$ if and only if the marked points $\{j, a\}$ are on the same branch of $T$ at $v_n$.
     \item Let $C'$ be the $\PP^1$ component of $C$ containing the marked point $n$, and choose a representative marked point from each branch at $v_n$.  Let $a,i_1,\ldots,i_k,n$ be these representatives from least to greatest, and let $C'$ be the $\PP^1$ component containing $n$ in $C$.  Then, viewing $C'$ as an element of $M_{0, \{a,i_1, \ldots, i_k, n\}}$, where we label each special point of $C'$ by its representative $i_j$, we have $\psi_n(C') = [t_{i_1}:\ldots:t_{i_k}]$, where $\psi_n$ is the smaller Kapranov map.  (The coordinates $t_{i_j}$ may be computed by Equation \eqref{eq:kap-interior-2}.)
  \end{enumerate}
\end{prop}

\begin{example}
Let $C$ be the curve in Figure \ref{fig:branches}, and let $C'$ be the component containing marked point $4$.  Parameterize $C'\cong \PP^1$ such that branch $\{4\}$ is at $\infty$, branch $\{a,b,c\}$ is at $0$, and $\{2\}$ and $\{1,3\}$ are at $t$ and $s$, respectively. Picking representatives $i_1=1$ and $i_2=2$, we have $\psi_4(C') = [s:t]$, and so $\psi_4(C)=[0:0:s:t:s].$
\end{example}

To define $\Omega_n$, we can combine the $\psi$ and forgetting maps as follows to form a map:
$$\psi_n\times \pi_n: \Mbar_{0,S}\injto \PP^n\times \Mbar_{0,S\setminus n}$$
This map is known to be an embedding \cite{Ka1}, and we may iterate the construction on $\Mbar_{0,S\setminus n}$ and so on, obtaining the \textbf{iterated Kapranov embedding}
\[\emb_n: \Mbar_{0,S} \hookrightarrow \mathbb{P}^1 \times \mathbb{P}^2 \times \cdots \times \mathbb{P}^n.\]
Keel and Tevelev \cite{KeT} first examined this full embedding as part of studying the log canonical embedding of $\Mbar_{0,n}$, which is the composition of $\Omega_n$ with the Segre embedding $\mathbb{P}^1 \times \cdots \times \mathbb{P}^n \hookrightarrow \mathbb{P}^{(n+1)!-1}$.
The $i$-th factor of $\Omega_n$ is given by forgetting the points $p_{i+1}, \ldots, p_n$, then applying the Kapranov morphism $\psi_i$ on the smaller moduli space.  

\begin{example}
  If $C$ is the curve in Figure \ref{fig:branches}, we have $$\Omega_4(C)=([0:1],[0:0:1],[0:0:1:0],[0:0:s:t:s]).$$
\end{example}

\subsection{The Monin--Rana equations}\label{sec:MR}

For the purpose of notational consistency with formulas like \eqref{eq:kap-interior-2}, we write $$\left[x_b^{(i)}:x_{c}^{(i)}:x_1^{(i)}:x_2^{(i)}:\cdots:x_{i-1}^{(i)}\right]$$ for the coordinates of the factor $\PP^i$ in the product $\PP^1\times \PP^2 \times \cdots \times \PP^n$.  Now, for each pair of indices $i<j$ in $\{1,2,\ldots,n\}$, consider the $2\times (i+1)$ matrix 
\begin{equation}\label{eq:MR}
    \mathrm{Mat}_{i,j}:=\begin{pmatrix}
x_b^{(i)}(x_b^{(j)}-x_{i}^{(j)}) & x_c^{(i)}(x_c^{(j)}-x_{i}^{(j)}) & x_1^{(i)}(x_1^{(j)}-x_{i}^{(j)}) & \cdots & x_{i-1}^{(i)}(x_{i-1}^{(j)}-x_{i}^{(j)}) \\
x_b^{(j)} & x_c^{(j)} & x_1^{(j)} & \cdots & x_{i-1}^{(j)}
\end{pmatrix}.
\end{equation}

\begin{definition}
  The ideal $I_n$ is the homogeneous ideal generated by all $2\times 2$ minors of the matrices \eqref{eq:MR}.  We denote the subscheme cut out by $I_n$ in $\PP^1\times \PP^2 \times \cdots \times \PP^n$ by $\MR_n$, and we refer to the equations defined by setting the $2\times 2$ minors to $0$ as the \textbf{Monin--Rana equations}.
\end{definition}

\begin{example}
  In the case $n=2$, there is only one pair of indices $i<j$, namely $i=1$ and $j=2$.  In this case, $\Mbar_{0,S}\cong \Mbar_{0,5}$ and the unique $2\times 2$ matrix yields the equation in Example \ref{ex:M05}.
\end{example}

\begin{example}
  For $n=3$, as a shorthand we rename the coordinates $$[x_b:x_c], [y_b:y_c:y_1],[z_b:z_c:z_1:z_2]$$ in this case.  Then we have three matrices of the form \eqref{eq:MR}, $\mathrm{Mat}_{1,2}$, $\mathrm{Mat}_{1,3}$, and $\mathrm{Mat}_{2,3}$:
  \[
  \begin{pmatrix}
   x_b(y_b-y_1) & x_c(y_c-y_1) \\
   y_b & y_c
  \end{pmatrix},
  \begin{pmatrix}
   x_b(z_b-z_1) & x_c(z_c-z_1) \\
   z_b & z_c
  \end{pmatrix},
  \begin{pmatrix}
     y_b(z_b-z_2) & y_c(z_c-z_2) & y_1(z_1-z_2) \\
   z_b & z_c & z_1
  \end{pmatrix}
  \]
  The ideal $I_3$ defining $\Mbar_{0,S}\cong \Mbar_{0,6}$ is then generated by the five resulting $2\times 2$ minors. Note that the codimension of $\Mbar_{0, 6}$ is only $3$.
\end{example}

\section{Set-theoretic equality}\label{sec:set-theoretic}

We now prove that $\MR_n= \Omega_n(\Mbar_{0,S})$ as sets.

\subsection{The interior case}

We first consider the coordinates corresponding to points in the interior of the moduli space, that is, the irreducible stable curves.

\begin{prop}\label{prop:interior}
  If a point $\vec{x}$ in $\PP^1\times \PP^2 \times \cdots\times \PP^n$ satisfies the Monin--Rana equations and its $\PP^n$ coordinates $x^{(n)}_b,\ldots,x^{(n)}_{n-1}$ are all nonzero and distinct, then $\vec{x}$ is in the image of $\Omega_n$.  Moreover, in this case $\vec{x}=\Omega_n(C)$ for a curve $C$ in the interior $M_{0,S}$ of $\Mbar_{0,S}$.
\end{prop}

\begin{proof}
We first claim that for each $i=1,\ldots,n$, the coordinates $x^{(i)}_b,\ldots,x^{(i)}_{i-1}$ in the $\PP^i$ factor are all nonzero and distinct.  To show this, we show that if the coordinates in $\PP^i$ are distinct and nonzero then those in $\PP^{i-1}$ are.  Then the claim will follow by induction, since we assumed the $\PP^n$ coordinates are nonzero and distinct.

As a shorthand, let $$[y_b:y_c:y_1:\cdots:y_{i-1}]:=[x_b^{(i)}:x_c^{(i)}:x_1^{(i)}:\cdots:x_{i-1}^{(i)}]$$ be the $\PP^i$ coordinates of $\vec{x}$, and let $$[x_b:x_c:x_1:\cdots:x_{i-2}]:=[x_b^{(i-1)}:x_c^{(i-1)}:x_1^{(i-1)}:\cdots:x_{i-2}^{(i-1)}]$$ be the $\PP^{i-1}$, with the scaling chosen in each so that the leftmost nonzero entries of each are $1$.
Then the matrix $\mathrm{Mat}_{i-1,i}$ from Equation \eqref{eq:MR}: \begin{equation}\label{eq:matrix}
    \mathrm{Mat}_{i-1,i}=\begin{pmatrix}
x_b(y_b-y_{i-1}) & x_c(y_c-y_{i-1}) & x_1(y_1-y_{i-1}) & \cdots & x_{i-2}(y_{i-2}-y_{i-1}) \\
y_b & y_c & y_1 & \cdots & y_{i-2}
\end{pmatrix}.
\end{equation}

Since $\vec{x}$ satisfies the Monin--Rana equations, this matrix has rank $1$.  Therefore, since the bottom row is all nonzero and distinct by assumption, the top row is a scalar multiple $\lambda$ of it.  We first show that $\lambda\neq 0$.  If $\lambda=0$ then since each $y_j-y_{i-1}$ is nonzero (since the $y$'s are distinct), we have that each $x_j=0$ for $j=b,c,1,\ldots,i-1$, but then this means that all of the $x$ coordinates are $0$, which is impossible in projective coordinates.  Thus, $\lambda\neq 0$.

Now, we can solve for each $x_j$ using the equation $$x_j(y_j-y_{i-1})=\lambda y_j$$
which yields $$x_j=\frac{\lambda y_j}{y_j-y_{i-1}}.$$  We claim that the function $f(y)=\frac{\lambda y}{y-y_{i-1}}$ is injective; indeed, if $f(a)=f(b)$ then $\lambda a (b-{y_{i-1}})=\lambda b (a-y_{i-1})$, and simplifying yields $a=b$ (since both $\lambda$ and $y_{i-1}$ are nonzero).  Since all $y_j$'s are distinct, we conclude that the outputs $x_j=f(y_j)$ are all distinct (and nonzero) as well.

Now that we have proven the claim, we construct the point $C$ in the interior inductively.  Assume that there is a curve $C'$ in $M_{0,S\setminus n}$ for which $\Omega_{n-1}(C')$ equals the first $n-1$ tuples of projective coordinates of $\vec{x}$.  In particular, let $i=n$ in the above calculation, so that $x_b,\ldots,x_{n-2}$ are the $\PP^{n-1}$ coordinates and $y_b,\ldots,y_{n-1}$ are the $\PP^n$ coordinates.   Then $x_b$ is nonzero, so we can choose a scaling so that $x_b=1$.   Coordinatize $C'$ accordingly so that marked point $a$ is at $0$, marked point $n-1$ is at $\infty$, and marked point $b$ is at $1$.

With respect to this coordinatization, define $C$ to be formed from $C'$ by inserting $n$ at coordinate $\lambda$, where $\lambda$ is the ratio of the top row of the matrix to the bottom (which is now fixed because we scaled so that $x_b=1$).  Then the $\PP^n$ coordinates of $C$ under the Kapranov map $\psi_n$ are precisely the coordinates of $b,c,1,\ldots,n-1$ obtained by changing coordinates so that $n$ is now at $\infty$ rather than $n-1$.  But the M\"obius transformation that changes the coordinates back is precisely $f(y)=\frac{\lambda y}{y-y_{i-1}}$, because we have $f(\infty)=\lambda$ and $f(0)=0$ and $f(y_{i-1})=\infty$.  Thus, $y_b,\ldots,y_{n-1}$ are indeed the final coordinates of $C$ under the embedding.
\end{proof}

\subsection{General setup}

We now give a sketch of our proof of the set-theoretic equality, which we carry out in the remainder of this section by induction on $n$.  Let $\vec{x}\in \PP^1\times \PP^2\times \cdots \times \PP^n$ satisfy the Monin--Rana equations, and rename the $\PP^n$ coordinates of $\vec{x}$ to be $[y_b:y_c:y_1:\cdots:y_{n-1}]$.   Let $\vec{x}'$ consist of the first $n-1$ coordinate vectors (i.e., all but the $y$ coordinates).  Since $\vec{x}'$ still satisfies the Monin--Rana equations for $n-1$, inductively we assume there is a curve $C'$, with dual tree $T'$, corresponding to $\vec{x}'$ under $\Omega_{n-1}$.  

Now, we consider the values of $[y_b:y_c:y_1:\cdots:y_{n-1}]$ and color the marked points of $C'$ as follows: we color marked point $i$ by color `$Z$' if $y_i=0$ (or if $i=a$), and we choose another color for each distinct nonzero coordinate value, so that marked points $j$ and $k$ have the same color if and only if $y_j=y_k$.  We will show that the Monin--Rana equations force each color to be separated from all other colors by an edge, and that this will determine uniquely where the marked point $n$ must be inserted in $C'$ in order to create a curve $C$ that maps to $\vec{x}$.

We therefore begin with some combinatorial definitions and lemmata about coloring trees.

\subsection{Strong separation}
For any tree $T$ corresponding to a curve $C\in \Mbar_{0,S}$, and any $i$, we write $$T|_i:=\pi_{i+1}\circ \cdots \circ \pi_{n}(T)$$ to denote the tree formed by forgetting all marked points greater than $i$.  That is, $T|_i$ is the dual tree of $\pi_{i+1}\circ \cdots \circ \pi_{n}(C)$.

\begin{definition}
On a leaf-labeled tree, we say leaf $i$ \textbf{separates} leaves $j$ and $k$ if the internal vertex $v_i$ adjacent to the leaf edge $i$ lies on the unique path from $j$ to $k$ in the tree.  In other words, $j$ and $k$ are on different branches at $v_i$.
\end{definition}

\begin{definition}\label{def:insertions}
Let $T$ be an at-least-trivalent tree on vertices labeled $a,b,c,1,\ldots,n$ with a coloring of every leaf either $R$ or $G$ (for `red' or `green') such that $a$ is colored $R$.  

We say $T$ has an \textbf{$R$-type bad configuration} if there exist indices $i,j,k\in \{b,c,1,\ldots,n\}$ with $j,k<i$ such that $j,k$ are colored $G$, $i$ is colored $R$, and in the tree $T|_i$, the leaf $i$ separates $j$ and $k$.  We also say $T$ has a \textbf{$G$-type bad configuration} if there exist $i,j,k\in \{b,c,1,\ldots,n\}$ with $j,k<i$ such that $i,k$ are colored $G$, $j$ is colored $R$, and in $T|_i$, the leaf $i$ separates $a$ and $j$.   (See Figure \ref{fig:bad-insertions} for an illustration of each type of bad configuration.)

Finally, we say $T$ has the \textbf{strong separation property} if it does not have any $R$-type or $G$-type bad configurations.  We also say $R$ and $G$ are \textbf{strongly separated} from one another on $T$.
\end{definition}

\begin{figure}
    \centering
    \includegraphics{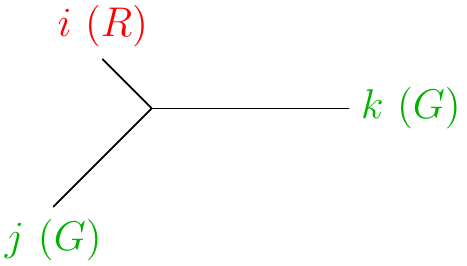} \hspace{1.5cm} \includegraphics{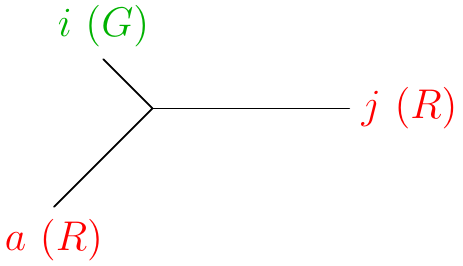}
    \caption{At left, an $R$-type bad configuration, which also has an $R$-colored vertex $a$ (not shown) which may branch out anywhere along the path from $j$ to $k$. At right, a $G$-type bad configuration, which similarly has a $G$-colored vertex $k$.}
    \label{fig:bad-insertions}
\end{figure}

\begin{remark}
Since the marked point $a$ will be colored $Z$ in the fully colored tree, the two scenarios in which we will be applying the strong separation property are:
\begin{enumerate}
    \item When $R=Z$ and the color $G$ represents the union of all nonzero colors,
    \item When $G$ is a nonzero color and $R$ represents the union of all colors besides $G$.
\end{enumerate}
\end{remark}

We now prove the key lemma about strong separation.

\begin{lemma}\label{lem:strongsep}
Suppose $T$ is an at-least-trivalent tree satisfying the strong separation property, with at least one $R$ and at least one $G$.  Then there is a unique edge $e$ that separates all the $R$ leaves from the $G$ leaves.
\end{lemma}

\begin{proof}
  We assume the leaves of $T$ are labeled $a,b,c,1,2,\ldots,n$, and proceed by induction on $n$, starting with $n=0$.  
  
  For the base case, consider the tree having only $a,b,c$ as leaves attached at a single central vertex.  If all of $a,b,c$ are colored $R$ there is nothing to show, since we assume there is at least one $G$.  Otherwise, if only one of $b$ or $c$ (say $b$) is colored $G$ then the unique edge is $b$'s leaf edge.  Finally, if $b$ and $c$ are both colored $G$ then the unique edge is $a$'s leaf edge.  
  
  Now, for induction, suppose the claim holds for $n-1$, and let $T' = T|_{n-1}$. If all of the leaves of $T'$ are colored $R$, then all but $n$ in $T$ are colored $R$ and so the leaf edge $n$ is the unique edge separating $R$ from $G$.  Otherwise, $T'$ has at least one $G$ and at least one $R$, so by the inductive hypothesis we can let $e_0$ be the edge in $T'$ separating the $G$s from the $R$s.  We think of $T$ as being formed from $T'$ by inserting a new leaf edge with leaf label $n$, where the new edge attaches  to an existing vertex or edge of $T'$.
  
  \textbf{Case 1.} Suppose $n$ is colored $R$.  Since there is no $R$-type bad configuration, $n$ cannot be inserted into $T'$ along a path connecting two $G$s to form $T$.  Thus, if there are at least two $G$s in the tree so far, $n$ cannot be inserted anywhere on the $G$ side of edge $e_0$.  Therefore, $n$ is inserted either on the $R$ side or directly on edge $e_0$ itself, and so either $e_0$ or its new half-edge on the $G$ side still separates all $G$s from $R$s in $T$.
  
  If instead there is only one $G$ in $T'$, then $e_0$ is a leaf edge with that $G$ being its leaf, and inserting the $n$ anywhere similarly results in a leaf edge that separates the unique $G$ from the $R$s.
  
  \textbf{Case 2.} Suppose $n$ is colored $G$.  If all of the leaves of $T'$ are colored $R$, then wherever $n$ is inserted, the edge separating $R$ from $G$ in $T$ is the leaf edge $n$.
  
  Otherwise, suppose $T'$ has a leaf colored $G$.   Since there is no $G$-type bad configuration, $n$ cannot be inserted along a path from $a$ to any other $R$ in $T'$.  If $n$ were inserted on the $R$ side of $e_0$, then it must be at a (possibly new) vertex $v$ along a path between two $R$s, say $j_1$ and $j_2$.  Deleting vertex $v$ from $T$ separates $j_1$ and $j_2$ onto different connected components, so at least one of these components, say that of $j_1$, doesn't contain $a$.  Then $v$ is also on the path from $a$ to $j_1$, so $j_1,a,n$ gives a bad $G$-type configuration, a contradiction.

  Therefore, $n$ is either inserted on the $G$ side of $e_0$ or on $e_0$ itself, and the claim follows.
\end{proof}

We can use the lemma above to conclude a more general result about multi-colored trees in which every color is strongly separated from the union of the others, as follows (and as illustrated in Figure \ref{fig:shrink}). 
\begin{corollary}\label{cor:shrink-colors}
  Suppose $T$ is an at-least-trivalent tree with each leaf assigned one of several possible colors.  Suppose that each color is strongly separated from the union of the remaining colors.  Then if we replace each monochromatic branch of $T$ with a single leaf of that color, we obtain an at-least-trivalent tree $T'$ in which each leaf is a different color.
\end{corollary}

\begin{proof}
  By the strong separation property on each color, each color is separated from the rest of the tree in $T$ by a unique edge by Lemma \ref{lem:strongsep}, and therefore is on a single branch of $T$ starting from that edge.  Replacing each branch by the leaf of its color, we therefore obtain an at-least-trivalent tree $T'$ with no repeated colors.
\end{proof}

\begin{figure}
    \centering
    \includegraphics{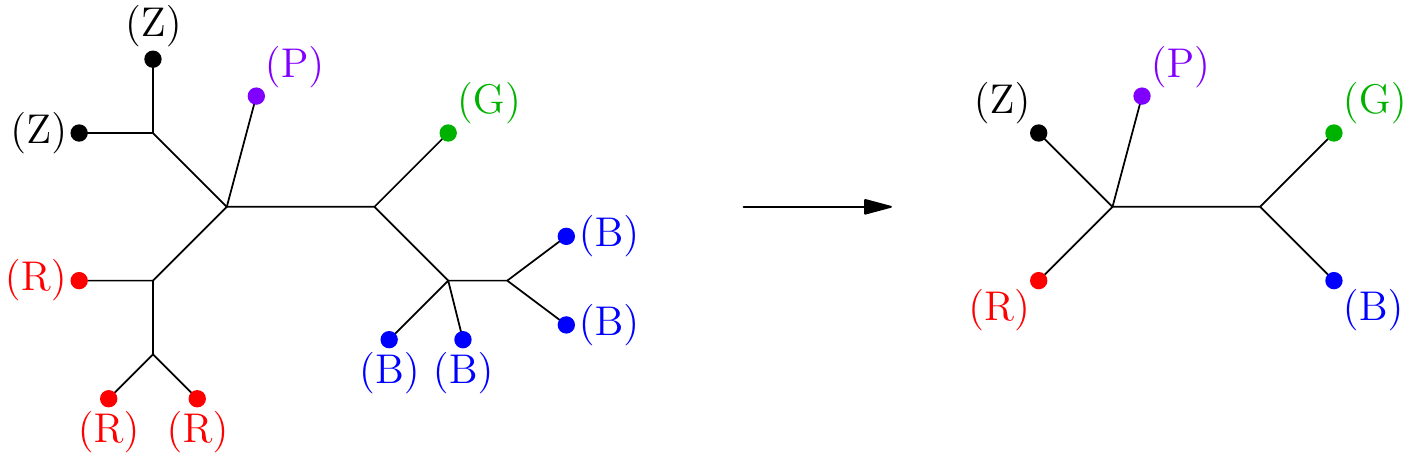}
    \caption{An illustration of Corollary \ref{cor:shrink-colors}.  At left, a tree $T$ in which each vertex is colored either red (R), green (G), blue (B), purple (P), or black (Z), and such that every color is separated from the remaining colors by a unique edge.  At right, the tree $T'$ formed by replacing the branch defined by each color with a single leaf.}
    \label{fig:shrink}
\end{figure}

\subsection{Equations implying strong separation}

We now show that if $\vec{x}$ satisfies the Monin--Rana equations, then the points colored `$Z$' (corresponding to $y_i=0$) on the associated curve $C'$ to $\vec{x}'$ are separated from the other colors by an edge.  

\begin{lemma}\label{lem:eqns-imply-strong-separation}
Let $C'$ be a curve in $\Mbar_{0,S\setminus n}$ and let $\vec{x}'\in \PP^1\times \cdots \times \PP^{n-1}$ be its associated coordinates.  Let $y=[y_b:y_c:y_1:\cdots:y_{n-1}]$ be a point in $\PP^n$, and suppose the combined vector of coordinates $(\vec{x}',y)\in \PP^1\times \cdots \times \PP^n$ satisfies the Monin--Rana equations.  If we color the leaves of the dual tree $T'$ by the rule:
  \begin{itemize}
      \item Color leaf $a$ by $R$,
      \item Color leaf $i$ by $R$ if $y_i=0$,
      \item Color leaf $i$ by $G$ if $y_i\neq 0$,
  \end{itemize}
  then this coloring of $T'$ satisfies the strong separation property.
\end{lemma}

\begin{proof}
  Assume for contradiction that it does not satisfy the strong separation property.  Then there is either a $G$-type or $R$-type bad configuration.
  
  \textbf{Case 1.} Suppose there is an $R$-type bad configuration, with indices $i,j,k$ colored $R,G,G$ respectively, where $j,k<i$ and leaf edge $i$ is attached to the path from $j$ to $k$ in $T'|_i$.  Thus, if  $[x_b:x_c:\cdots:x_{i-1}]$ are the $\PP^i$ coordinates of $\vec{x}'$, we have that $x_j\neq x_k$ by Proposition~\ref{prop:PsiCoordinates}(1) applied to $\psi_i$.
  
  By the definition of the coloring we also have $y_i=0$, and $y_j,y_k\neq 0$.  Thus, the Monin--Rana $2\times 2$ minor given by indices $j,k$ between the $x$ and $y$ variables is:
  
  $$\det \begin{pmatrix}
  x_j(y_j-y_i) & x_k(y_k-y_i) \\
  y_j & y_k
  \end{pmatrix}=
  \det \begin{pmatrix}
  x_j(y_j) & x_k(y_k) \\
  y_j & y_k
  \end{pmatrix}$$
  which is nonzero because $y_k,y_j\neq 0$ and $x_j\neq x_k$. This contradicts the assumption that $(\vec{x}',y)$ satisfies the Monin--Rana equations.
  
  \textbf{Case 2.}  Suppose there is a $G$-type bad configuration, with indices $i,j,k$ colored $G,R,G$ respectively, with $j,k<i$ and leaf edge $i$ is attached to the path from $a$ to $j$ in $T'|_i$.  Then in the coordinates $[x_b:x_c:\cdots:x_{i-1}]$ of $\PP^i$, we have $x_j\neq 0$ by Proposition~\ref{prop:PsiCoordinates}(2) applied to $\psi_i$.  Furthermore, because of the labeling, we have $y_j=0$ and $y_i,y_k\neq 0$.  Thus, in this case the relevant $2\times 2$ matrix is
  
    $$\begin{pmatrix}
  x_j(y_j-y_i) & x_k(y_k-y_i) \\
  y_j & y_k
  \end{pmatrix}=
  \begin{pmatrix}
  -x_jy_i & x_k(y_k-y_i) \\
  0 & y_k
  \end{pmatrix}$$
  whose determinant is $-x_jy_iy_k$, which is nonzero, a contradiction.  Therefore, the coloring of $T'$ satisfies the strong separation property.
\end{proof}

We now show that, similarly, the Monin--Rana equations force any other color $G$ to be separated from the rest.

\begin{lemma}\label{lem:eqns-imply-general-strong-separation}
Let $C'\in \Mbar_{0,S\setminus n}$, and let $\vec{x}'\in \PP^1\times \cdots \times \PP^{n-1}$ be its associated coordinates.  Let $y=[y_b:y_c:y_1:\cdots:y_{n-1}]$ be a point in $\PP^n$, and suppose the combined vector of coordinates $(\vec{x}',y)\in \PP^1\times \cdots \times \PP^n$ satisfies the Monin--Rana equations.  Let $\beta$ be one of the nonzero coordinate values among $y_b,y_c,\ldots,y_{n-1}$ (under some chosen scaling).  Then if we color the leaves of the dual tree $T'$ of $C'$ by the rule:
  \begin{itemize}
      \item Color leaf $a$ by $R$,
      \item Color leaf $i$ by $R$ if $y_i\neq \beta$,
      \item Color leaf $i$ by $G$ if $y_i=\beta$,
  \end{itemize}
  then this coloring satisfies the strong separation property.
\end{lemma}

\begin{proof}
  Assume for contradiction that it does not satisfy the strong separation property.  Then there is either an $R$-type or $G$-type bad configuration.
  
  \textbf{Case 1.} Suppose there is an $R$-type bad configuration, with indices $i,j,k$ colored $R,G,G$ respectively, where $j,k<i$ and leaf edge $i$ attached to the path from $j$ to $k$ in $T'|_i$.  Thus, if we use the variables $[x_b:x_c:\cdots:x_{i-1}]$ for the $\PP^i$ coordinates, we have that $x_j\neq x_k$.  
  
  By the definition of the coloring we also have $y_i\neq \beta$, and $y_j,y_k=\beta$.  Thus, the Monin--Rana equation given by indices $j,k$ between the $x$ and $y$ variables looks like:
  
  $$\begin{pmatrix}
  x_j(y_j-y_i) & x_k(y_k-y_i) \\
  y_j & y_k
  \end{pmatrix}=
  \begin{pmatrix}
  x_j(\beta-y_i) & x_k(\beta-y_i) \\
  \beta & \beta
  \end{pmatrix}$$
  which has a nonzero determinant because $y_i\neq \beta$, $\beta\neq 0$, and $x_j\neq x_k$. This contradicts the assumption that the Monin--Rana equations are satisfied.
  
  \textbf{Case 2.}  Suppose there is a $G$-type bad configuration, with indices $i,j,k$ colored $G,R,G$ respectively, with $j,k<i$ and leaf edge $i$ attached to the path from $a$ to $j$ in $T'|_i$.  Then in the coordinates $[x_b:x_c:\cdots:x_{i-1}]$ of $\PP^i$, we have $x_j\neq 0$.  Furthermore, because of the labeling, we have $y_j\neq \beta$ and $y_i=y_k=\beta$.  Thus, in this case the relevant $2\times 2$ matrix is
  
    $$\begin{pmatrix}
  x_j(y_j-y_i) & x_k(y_k-y_i) \\
  y_j & y_k
  \end{pmatrix}=
  \begin{pmatrix}
  x_j(y_j-\beta) & 0 \\
  y_j & \beta
  \end{pmatrix}$$
  whose has nonzero determinant since $x_j\neq 0$, $\beta\neq 0$, and $y_j\neq \beta$.  This completes the proof.
\end{proof}

\subsection{Main proof}

We now prove the set-theoretic equality.  Our strategy will be to show, by induction on $n$, that if a tuple of coordinates $\vec{x}\in \PP^1\times \PP^2 \times \cdots \times \PP^n$ satisfies the Monin--Rana equations, then there is a curve $C\in \Mbar_{0,S}$ such that $\Omega_n(C)=\vec{x}$.  The base case, $n=1$, is immediate since there are no Monin--Rana equations and $\Omega_1:\Mbar_{0,abc1}\to \PP^1$ is an isomorphism.

For the induction, suppose the statement holds for $n-1$ and let $\vec{x}\in \PP^1\times \cdots \times\PP^n$ satisfy the Monin--Rana equations.  Let $C'$ be the curve corresponding to the first $n-1$ coordinates $\vec{x}'$ by the induction hypothesis, and let $T'$ be its at-least-trivalent dual tree.  Let $[y_b:y_c:\cdots:y_{n-1}]$ be the $\PP^n$ coordinates of $\vec{x}$, and color each leaf of $T'$ as follows:
\begin{itemize}
    \item Color leaf $a$ by $Z$ (for $0$)
    \item Color all other leaves $j$ by the value $y_j$ (where $Z$ represents the color $0$).
\end{itemize}

In order to construct the new curve $C$ from $C'$, we will find a place to insert the marked point $n$ by first proving two lemmas about this coloring.

\begin{definition}
We say the $y$ coordinates are \textbf{binary} if there is a rescaling of $[y_b:y_c:\cdots:y_{n-1}]$ such that each coordinate is either $0$ or $1$.
\end{definition}

\begin{lemma}\label{lem:monochromatic}
  If $y$ is a binary vector, then there is a unique edge $e$ that separates the two colors.  Otherwise, there is a unique vertex $v$ in $T'$ such that the branches at $v$ are each monochromatic, and no two of the branches have the same color.
\end{lemma}

\begin{proof}
By Lemma \ref{lem:eqns-imply-strong-separation}, the coloring of $T'$ defined above satisfies the strong separation property with $R=Z$, and $G=I$ being the union of all nonzero colors.  Therefore, there is a unique edge $e$ that separates the $Z$'s and $I$'s by Lemma \ref{lem:strongsep}.  

If there is only one nonzero color $I$, that is, if $y$ is a binary vector, then this edge $e$ separates the two colors, proving the first statement of the lemma.

  Now suppose $y$ is not binary, so there are at least three colors, and let $v$ be the $I$-side vertex of edge $e$.  Consider the tree $T''$ formed by replacing each monochromatic branch by a leaf, as in Corollary \ref{cor:shrink-colors}.  Note that in $T''$, there is one leaf colored $Z$ attached to $v$.   If $T''$ consists entirely of leaves attached to $v$, we are done; otherwise, by following an arbitrary path oriented away from $Z$ in the tree, we can find a vertex $v_0\neq v$ such that every branch at $v_0$ not containing the leaf colored $Z$ is a leaf in $T''$.  This means that, at the corresponding vertex $v_0$ in $T'$, each branch other than the branch containing the $Z$-colored leaves is monochromatic (and each has a different color).

 We will show that in fact the assumption that $v_0\neq v$ leads to a contradiction, as follows.
  Let $u$ and $t$ be the colors of the branches $B_u, B_t$ at $v_0$ in $T'$ with the two smallest minimal labels among the branches away from $Z$. Let $m_u$ and $m_t$ be the minimal labels of $B_u$ and $B_t$, respectively, and assume without loss of generality that $m_u>m_t$.  We consider two cases:
  
  \textbf{Case 1:} Suppose there exists a leaf colored some nonzero color $s\neq u,t$ whose label $k$ satisfies $k<m_u$.  In this case, for simplicity of notation we set $i=m_u$ and $j=m_t$.  Then $y_i=u$, $y_j=t$, and $y_k=s$.  
  
  Since $i>j$ and $i=m_u$ is minimal in $B_u$, we know that in $T'|_i$, $i$ separates $j$ from the $Z$ branch and so we see that $x_j^{(i)}\neq 0$. Moreover, since $k$ has a different color than $i$ or $j$, and we assumed $i,j$ are the two smallest minimal elements of all the branches at $v_0$ not containing $Z$, and $i > k$, then $k$ is in the branch at $v_0$ that contains the $Z$-colored leaves. Therefore, in $T'|_i$, $i$ does not separate $k$ from the $Z$ branch, so $x_k^{(i)}=0$.  Thus, the $2\times 2$ matrix
  $$\begin{pmatrix}
  x_j^{(i)}(y_j-y_i) & x_k^{(i)}(y_k-y_i) \\
  y_j & y_k
  \end{pmatrix}=
  \begin{pmatrix}
  x_j^{(i)}(t-u) & 0 \\
  y_j & s
  \end{pmatrix}$$
  has determinant $x_j^{(i)}(t-u)s$, which is nonzero by the discussion above.  This gives a contradiction in this case.

  \textbf{Case 2:} Suppose there is no leaf as in Case 1, that is, all leaves of all other colors $s\neq u,t$ are larger than $m_u$.  In this case, for simplicity of notation, we let $j=m_u$ and $k=m_t$, and let $i$ be the minimal element of any other nonzero branch $B$ at the original vertex $v$.  Then $i>j,k$ and $i$ separates $j,k$ from $a$ in $T|_i$.
  
  Now we have $y_j=t$, $y_k=u$, $y_i=s$, and we also know $x_k^{(i)}=x_j^{(i)}\neq 0$ (since $v\neq v_0$).  Then we have the $2\times 2$ matrix
  $$\begin{pmatrix}
  x_j^{(i)}(y_j-y_i) & x_k^{(i)}(y_k-y_i) \\
  y_j & y_k
  \end{pmatrix}=
  \begin{pmatrix}
  x_j^{(i)}(t-s) & x_j^{(i)}(u-s) \\
  t & u
  \end{pmatrix}$$
  whose determinant is $$x_j^{(i)}u(t-s)-x_j^{(i)}t(u-s)=x_j^{(i)}ts-x_j^{(i)}us=x_j^{(i)}s(t-u)\neq 0$$ and again we have a contradiction.
\end{proof}

We can now prove the main result.

\begin{thm}\label{thm:set-theoretic}
Suppose a point $\vec{x}$ in $\PP^1\times \PP^2 \times \cdots \times \PP^n$ satisfies the Monin--Rana equations. Then there is a curve $C\in \Mbar_{0,abc1\cdots n}$ such that $\Omega_n(C)=\vec{x}$.
\end{thm}

\begin{figure}
    \centering
\begin{tikzcd}
C\in \Mbar_{0,S}\arrow[r,"\pi_{\vec{i},n}"]\arrow[d]&\Mbar_{0,\{a,i_1,\dots,i_m,n\}}\arrow[d]&\arrow[l,hook'] M_{0,\{a, i_1,\dots,i_m,n\}}\arrow[d,"\pi_n"] \ni C'''\\
C'\in \Mbar_{0,S\setminus n}\arrow[r,"\pi_{\vec{i}}"] & \Mbar_{0,\{a,i_1,\dots,i_m\}}&\arrow[l,hook'] M_{0,\{a,i_1,\dots,i_m\}}\ni C''&
\end{tikzcd}
    \caption{We construct the curve $C$ from $C'$ (via $C''$ and $C'''$) in the proof of Theorem \ref{thm:set-theoretic}.}
    \label{fig:Cdiagram}
\end{figure}

\begin{proof}
 Using the inductive setup above and the notation of Lemma \ref{lem:monochromatic}, recall that $\vec{x} = \vec{x}'\times y$, where $y$ are the $\bP^n$ coordinates of $\vec{x}$, and $C'\in \overline{M}_{0,S\setminus n}$ such that $\Omega_{n-1}(C') = \vec{x}'$.  Let $T$ be the tree formed from $T'$ by:
  \begin{enumerate}
      \item If $y$ is a binary vector, insert $n$ in the middle of the edge $e$.
      \item Otherwise, insert $n$ at the vertex $v$.
  \end{enumerate}
  Note that if the desired curve $C$ with $\Omega_n(C)=\vec{x}$ exists then it must have dual tree $T$ in order for all the $Z$-side leaves to correspond to the zero coordinates of the $y$ vector. We now construct the curve $C$.
  
  In case (1), we are done since the curve $C$ is uniquely defined (since the marked point $n$ is inserted at a node of the curve) and maps to the specified coordinates under $\Omega_n$.
  
  In case (2), we show there is a unique place to insert the marked point $n$ on the component corresponding to vertex $v$ to form the curve $C$, as follows.  Let $a < i_1<\cdots < i_m$ be the minimal labels of the branches at $v$ in $T'$.  We put
\begin{align*}
\pi_{\vec{i}, n} &: \Mbar_{0,S} \to \overline{M}_{0,\{a,i_1,\dots, i_m, n \}}, \\
\pi_{\vec{i}} &: \Mbar_{0,S\setminus n} \to \overline{M}_{0,\{a,i_1,\dots, i_m \}}, \\
\pi &: \mathbb{P}^{n} \dashrightarrow \mathbb{P}^{m-1},
\end{align*}
where $\pi_{\vec{i}, n}$ and $\pi_{\vec{i}}$ are the forgetting maps obtained by forgetting all other marked points besides $a,i_1,i_2,\ldots,i_m$ (and $n$), and $\pi$ is the rational map
\[\pi([y_b : y_c : y_1 : \cdots: y_{n-1}]) = [y_{i_1} : \cdots : y_{i_m}].\]
(See Figure \ref{fig:Cdiagram}.) Note that if $C \in \Mbar_{0,S}$ is obtained by inserting $n$ somewhere on the component of $C'$ corresponding to $v$, then
\[\psi_{n}(\pi_{\vec{i},n}(C)) = \pi(\psi_{n}(C))\] 
holds by definition since $[y_{i_1} : \cdots : y_{i_j}]$ is not the zero vector.

We will construct $C$ via two auxiliary curves $\localcurve$ and $\localcurvewithpn$ (see Figure \ref{fig:Cdiagram}). First, let $\localcurve = \pi_{\vec{i}}(C')$. In fact, $\localcurve$ is in the interior $M_{0,\{a,i_1, \ldots, i_m\}}$. In addition, since each $i_j$ is minimal in its branch, $i_j$ is adjacent to $v$ in the tree after forgetting the marked points $> i_j$. Therefore, $\Omega_{m-2}(\localcurve) \in \mathbb{P}^1 \times \cdots \times \mathbb{P}^{m-2}$ is the point whose $\PP^{j-2}$ coordinates (for $j\ge 3$) are the restriction of the $\PP^{i_j}$ coordinates of $\vec{x}$ to indices $i_1,\dots, i_{j-1}$. 

Let $\localyvector = \pi(y)$; this is well defined since by hypothesis $y_{i_1}, \ldots, y_{i_m}$ are not all zero.  Note that the point $\Omega_{m-2}(\localcurve)\times \localyvector$ in $\PP^1\times\cdots\times \PP^{m-1}$ satisfies the Monin--Rana equations;
    indeed, the equations involving $\localyvector$ for this point are the $2\times 2$ minors
    \[\det
    \begin{pmatrix}
    x_{i_j}(y_{i_j}-y_{i_m}) & x_{i_k}(y_{i_k}-y_{i_m})\\
    y_{i_j} & y_{i_k}
    \end{pmatrix}
    \]
but these are just a subset of the Monin--Rana equations for $\vec{x}$, which hold by assumption. Similar logic applies to the Monin--Rana equations involving only coordinates from $\vec{x}'$.
    
By Proposition~\ref{prop:interior}, since the coordinates of $\localyvector$ are nonzero and distinct by Lemma \ref{lem:monochromatic}, there exists a unique stable curve $\localcurvewithpn$ in the interior $M_{0,\{a,i_1,\dots,i_m,n\}}$ such that $\Omega_{m-1}(\localcurvewithpn) = \Omega_{m-2}(\localcurve)\times \localyvector$.  We now define $C \in \Mbar_{0,S}$ by inserting $n$ on the $v$ component of $C'$ at the point corresponding to $p_n \in \localcurvewithpn$, so that $\pi_n(C)=C'$ and $\pi(\psi_n(C)) = \psi_n(\localcurvewithpn)$. Therefore, up to a common nonzero scalar we have
\[
\psi_n(C)_{i_j} = \pi(\psi_n(C))_j = \psi_n(\localcurvewithpn)_j = \localyvector_j = y_{i_j}.\]
Since $n$ is inserted on the $v$ component of $C'$, we also have $\psi_n(C)_\ell = \psi_n(C)_{i_j} = y_{i_j} = y_\ell$ whenever $\ell$ is in the same branch as $i_j$ (by Lemma \ref{lem:monochromatic}). Therefore, $\psi_n(C) = y$ as required, so $\Omega_n(C) = \Omega_{n-1}(C')\times \psi_n(C) = \vec{x}'\times y = \vec{x}$.
\end{proof}

\section{Scheme-theoretic equality}\label{sec:scheme-theoretic}

Let $\MR_n \subseteq \PP^1 \times \cdots \times \PP^n$ be the subscheme defined by the Monin--Rana equations. Theorem \ref{thm:set-theoretic} and Proposition \ref{prop:inclusion} imply the set-theoretic equality $\Omega_n(\Mbar_{0,S}) = \MR_n$. In particular, $\MR_n$ is irreducible of dimension $n$.
We now turn to upgrading this to an equality of subschemes:

\begin{MainThm}
  We have $\Omega_n(\Mbar_{0,S}) = \MR_n$ as subschemes of $\PP^1 \times \cdots \times \PP^n$.
\end{MainThm}

We begin by making several reductions to set up the strategy for the proof below.  By Theorem \ref{thm:set-theoretic}, both schemes have the same closed points, so it is enough to show scheme-theoretic equality locally at each point $C \in \Mbar_{0,S}$. We have the scheme-theoretic inclusion 
\begin{equation} \label{eq:scheme-theoretic-subseteq}
    \Omega_n(\Mbar_{0,S}) \subseteq \MR_n
\end{equation}
by Proposition \ref{prop:inclusion}. We wish to show the reverse inclusion. Since $\Mbar_{0,S}$ is smooth and $\Omega_n$ is an embedding, it is enough to show the equality of Zariski tangent spaces
\begin{equation}
\label{eq:tg-spaces}
    T_{\Omega_n(C)}\Omega_n(\Mbar_{0,S}) = T_{\Omega_n(C)}\MR_n,
\end{equation}
which will imply that $\MR_n$ is reduced. The $(\subseteq)$ direction holds by Equation \eqref{eq:scheme-theoretic-subseteq}, so it is enough to show
\begin{equation}\label{eq:dim-tg-ineq}
\dim T_{\Omega_n(C)}\MR_n \leq \dim T_{\Omega_n(C)}\Omega_n(\Mbar_{0,S}) = n.
\end{equation}
We demonstrate the inequality \eqref{eq:dim-tg-ineq} by linearizing the Monin--Rana equations and counting the induced equations on the tangent space, and using induction on $n$. 
Let $\vec{x} = \Omega_n(C)$ and let 
\[\mathrm{pr}_n : \PP^1 \times \cdots \times \PP^{n-1} \times \PP^n \to \PP^1 \times \cdots \times \PP^{n-1}\] be the projection from the last factor.

We write $(\mathrm{pr}_n)_{\ast}$ for the induced map on tangent spaces at $\vec{x}$.  We have 
$\mathrm{pr}_n(\MR_n) \subseteq \MR_{n-1}$ since the generating equations of $\MR_{n-1}$ are among those of $\MR_n$, so in particular
\begin{equation}\label{eqn:containment}
T_{\vec{x}}\MR_n \subseteq
(\mathrm{pr}_n)_*^{-1} \big( T_{\mathrm{pr}_n(\vec{x})}\MR_{n-1} \big).
\end{equation}
The dimension of the larger space is
\begin{equation}\label{eqn:inequality}
\dim (\mathrm{pr}_n)_*^{-1} \big( T_{\mathrm{pr}_n(\vec{x})}\MR_{n-1} \big)
= \dim T_{\mathrm{pr}_n(\vec{x})}\MR_{n-1} + \dim \ker (\mathrm{pr}_n)_{\ast} \leq (n-1) + n,
\end{equation}
where the $(n-1)$ holds by induction. By Equations \eqref{eqn:containment} and \eqref{eqn:inequality}, in order to prove \eqref{eq:dim-tg-ineq} it suffices to exhibit $n-1$ additional equations on $T_{\vec{x}}(\PP^1\times \cdots \times \PP^n)$ cutting out $T_{\vec{x}}\MR_n$ as a subspace of $(\mathrm{pr}_n)_*^{-1} \big( T_{\mathrm{pr}_n(\vec{x})}\MR_{n-1} \big)$. In summary:

\begin{strategy}\label{strat:summary}
  To prove Theorem \ref{thm:main}, it suffices to exhibit $n-1$ additional equations on $T_{\vec{x}} (\PP^1\times \cdots\times \PP^n)$ cutting out $T_{\vec{x}}\MR_n$ in $(\mathrm{pr}_n)_*^{-1}(T_{\mathrm{pr}_n(\vec{x})}\MR_{n-1})$.
\end{strategy}

Letting $[y_b : y_c : y_1 : \cdots : y_n]$ be the coordinates on $\PP^n$, and $\y_b,\ldots,\y_n$ the corresponding generators of the cotangent space at $\vec{x}$, we will (mostly) find linear equations with distinct leading terms $\y_i$, which are then evidently transverse to $(\mathrm{pr}_n)_*^{-1}\big( T_{\mathrm{pr}_n(\vec{x})} \MR_{n-1} \big)$.

Combinatorially, these equations come from examining the branches of the dual tree of $C$ at the vertex $v_n$ that the leaf edge $n$ attaches to.  There is an exceptional case that behaves somewhat differently from the general case, namely when $v_n$ is trivalent ($\mathrm{deg}(v_n)=3$) with two non-leaf edges. In all other cases, we find $n-1$ additional equations that have distinct leading terms $\y_i$ (see Corollary \ref{cor:summary}).  In the exceptional case, we obtain $n-2$ equations involving the $\y_i$'s as above, plus one final equation, entirely in the coordinates of the factors $\PP^1 \times \cdots \times \PP^{n-1}$, but which we show does not vanish on $T_{\mathrm{pr}_n(\vec{x})}\MR_{n-1}$ (Proposition \ref{prop:extra} and Lemma \ref{lem:nonzero-coord}).

\begin{remark}[Geometric reason for the exceptional case] \label{rmk:why-exceptional}
The exceptional case corresponds to the failure of the forgetful morphism $\pi_n$ to be surjective on tangent spaces. A node $q \in C$ determines a tangent vector $t_q \in T_C\overline{M}_{0,S}$ by smoothing $q$ while retaining the relative positions of special points on each side, and $(\pi_n)_*(t_q) = 0$ if and only if $\pi_n$ contracts one of the components containing $q$. In the exceptional case, this gives \emph{two} independent tangent vectors that map to $0$, so $(\pi_n)_*$ is not surjective. In particular, we have a \emph{strict} inclusion
\[
(\Omega_{n-1})_* (\pi_n)_* (T_C\overline{M}_{0,S}) \subsetneq (\Omega_{n-1})_* (T_{\pi_n(C)}\overline{M}_{0,S \setminus n})
\]
of subspaces of the tangent space in $\mathbb{P}^1 \times \cdots \times \mathbb{P}^{n-1}$. Thus, there must be an additional linear condition in the coordinates of $\mathbb{P}^1 \times \cdots \times \mathbb{P}^{n-1}$, satisfied by the smaller subspace but not the larger one. This condition is then satisfied by $T_{\Omega_n(C)}\Omega_n(\overline{M}_{0,S})$ because $\Omega_{n-1} \circ \pi_n = \mathrm{pr}_n \circ \Omega_n$.
\end{remark}

We now proceed with the core of the proof.

\subsection{Tangent equations setup}

Fix $i<n$ and let $[y_b:y_c:y_1:\cdots:y_{n-1}]$ be the $\PP^n$ coordinates and $[x_b:\cdots :x_{i-1}]$ be the $\PP^{i}$ coordinates.  Then the Monin--Rana equation obtained from columns $m$ and $r$ of the $2\times (i+1)$ matrix $\mathrm{Mat}_{i,n}$ of \eqref{eq:MR} is:
\begin{equation}\label{eq:general-MR}
    x_m(y_m-y_i)y_r=x_r(y_r-y_i)y_m.
\end{equation}

Let $\vec{t}$ be a point in $\PP^1\times \cdots \times \PP^n$ satisfying all the Monin--Rana equations, with associated curve $C\in \Mbar_{0,S}$ by Theorem \ref{thm:set-theoretic}, and associated trivalent tree $T$.
Let $y_j=t_j$ be the $y$ coordinates of $\vec{t}$ and $x_j=s_j$ be the $x$ coordinates.  If $t_k$ is the leftmost nonzero among the $y$ coordinates of $\vec{t}$, we work in the chart where $y_k=1$, and similarly in all other $\PP^r$ factors.

We work locally at $\vec{t}$ as follows.
\begin{definition}\label{def:tangent}
Set local variables $\x_j=x_j-s_j$ and $\y_j=y_j-t_j$, which we view as elements of the Zariski cotangent space $\mathfrak{m}/\mathfrak{m}^2$, where $\mathfrak{m}$ is the maximal ideal of $\mathcal{O}_{\bP^1\times\cdots\times\bP^n,\vec{t}}$. Thus, all degree $2$ products of the variables $\x_j$ and $\y_j$ are zero.
\end{definition}

Then the Monin--Rana equation \eqref{eq:general-MR}, in the cotangent space, becomes:
\begin{equation}
    (\x_m+s_m)(\y_m-\y_i+t_m-t_i)(\y_r+t_r)=(\x_r+s_r)(\y_r-\y_i+t_r-t_i)(\y_m+t_m) \label{eqn:tangent-general}
\end{equation}

\begin{notation}[Chart index $k$]
In all of the following, set $k$ to be the chart index (setting $y_k=1$) in the $\PP^n$ variables, chosen as above to be the smallest $k$ such that $t_k\neq 0$.  In all of our diagrams, we will draw the path from $a$ to $k$ in the tree from left to right.   Note that, since $t_k\neq 0$, the leaf edge $n$ is directly attached to this path at vertex $v=v_n$ as shown in Figure \ref{fig:initial-tree}.
\end{notation}

\begin{figure}
    \centering
    \includegraphics{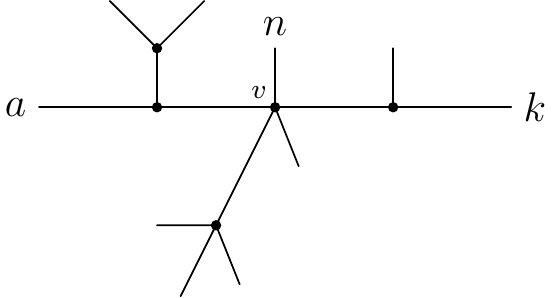}
    \caption{The tree $T$, with largest marked point $n$.  The value $k$ is the minimum element among all leaves that are not on the branch at $v$ containing $a$, and we choose $t_k=1$ as our chart in the $\PP^n$ coordinates.}
    \label{fig:initial-tree}
\end{figure}

The following lemma is not strictly necessary but is included for ease of exposition: 

\begin{lemma}
If $i<k$, then all the tangent equations coming from $\mathrm{Mat}_{i,n}$ are trivial, that is, they reduce to $0=0$.
\end{lemma}

\begin{proof}
 Consider the equation with parameters $m,r<i$. By assumption, we have $m,r,i<k$ and so by the definition of the chart we have $t_m=t_{r}=t_i=0$.  Then the $m,r$ tangent equation (\ref{eqn:tangent-general}) becomes $$(\x_m+s_m)(\y_m-\y_i)\y_{r}=(
 \x_k+s_{r})(\y_{r}-\y_i)\y_m.$$ 
 By the relations in Definition \ref{def:tangent}, the equation above is just $0 = 0$.
\end{proof}

Thus, we will hereafter assume $i \ge k$. We consider the cases $i=k$ and $i>k$ separately in the next two lemmas.

\begin{lemma}\label{lem:i=k}
  If $i=k$, then, renaming $m$ as $m_1$ and $r$ as $m_2$, the tangent equation \eqref{eqn:tangent-general} reduces to $$s_{m_1}\y_{m_2}=s_{m_2}\y_{m_1}.$$
\end{lemma}

\begin{proof}
Since $i$ is the chart we have $t_i=1$, and since $m_1,m_2<i$ we have $t_{m_1}=0$ and $t_{m_2}=0$. Substituting $m=m_1$ and $r=m_2$ and the above $t$ values into (\ref{eqn:tangent-general}), we have
$$(\x_{m_1}+s_{m_1})(\y_{m_1}-\y_{i}-1)\y_{m_2}=(\x_{m_2}+s_{m_2})(\y_{m_2}-\y_i-1)\y_{m_1}.$$ 
Expanding and cancelling quadratic terms gives $s_{m_1}\y_{m_2}=s_{m_2}\y_{m_1}$.
\end{proof}

\begin{remark}
We renamed $m$ and $r$ as $m_1$ and $m_2$ in the above Lemma statement as a notational convenience for later applications of Lemma \ref{lem:i=k}.
\end{remark}

\begin{lemma}\label{lem:i>k}
  If $i>k$, then the tangent equation \eqref{eqn:tangent-general} for $r=k$ is \begin{equation*}(s_m-s_k+s_kt_i)\y_m+(t_ms_k-s_m)\y_i=t_m(1-t_i)\x_k+(t_i-t_m)\x_m. \end{equation*}
\end{lemma}

\begin{proof}
Since $k$ is the chart, we have $\y_k=0$ and $t_k=1$, so after changing all $r$'s to $k$'s the tangent equation \eqref{eqn:tangent-general} becomes
$$(\x_m+s_m)(\y_m-\y_i+t_m-t_i)=(\x_k+s_k)(-\y_i+1-t_i)(\y_m+t_m).$$
Expanding this out, the constant terms vanish because the $s$ and $t$ values themselves satisfy the Monin--Rana equations. The linear term is then
$$(t_m-t_i)\x_m+s_m\y_m-s_m\y_i=t_m(1-t_i)\x_k-s_kt_m\y_i+s_k(1-t_i)\y_m.$$ Putting all the $y$'s on the left and $x$'s on the right yields the desired equation.
\end{proof}

\subsection{Equations along the \texorpdfstring{$a$}{`a'} branch}

We now use the general equations from Lemmas \ref{lem:i=k} and \ref{lem:i>k} to examine the tangent equations we obtain using indices of leaves on the branch at $v_n$ containing the leaf $a$.

Throughout all sections below, we write $$T|_i=\pi_{i+1}\circ \pi_{i+2}\circ \cdots \circ \pi_{n}(T)$$ as in Definition \ref{def:insertions}.
In other words, $T|_i$ is the tree formed by forgetting all leaves larger than $i$ in $T$, then stabilizing.

\begin{definition}\label{def:sub-branch}
In the tree $T$, let $v=v_n$ be the internal vertex that is adjacent to leaf $n$, and consider the path $P$ from $v$ to $a$.  Given another leaf $i$, let $e$ be the first edge on the path from $i$ to $a$ that is incident to some vertex $x$ of $P$. If $x$ is strictly between $v$ and $a$ on $P$, we say $i$ is on a \textbf{sub-branch off the $a$ branch} that \textbf{attaches at $x$}, and its entire \textbf{sub-branch} consists of all leaves whose path to $a$ also contains edge $e$.  The \textbf{distance} between $v$ and the sub-branch is the number of edges on the path from $v$ to $x$.
\end{definition}

The following picture will be useful in all of the proofs in this subsection.

\begin{center}
    \includegraphics{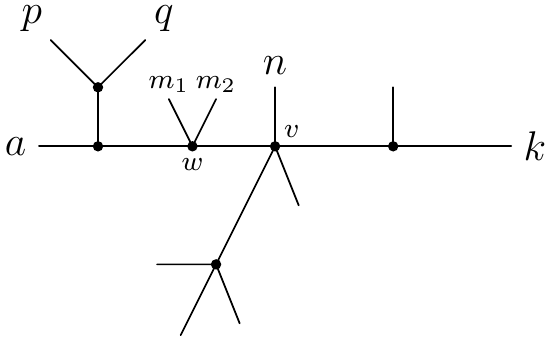}
\end{center}

In particular, we will always let $v=v_n$ be the internal vertex that is adjacent to leaf $n$, and note that $v$ lies on the path from $a$ to $k$ in $T$.

\begin{prop}[Two away from $v$ on $a$ branch]\label{prop:two-away-a} 
 In $T$, suppose $p$ is the minimal element of a sub-branch $B$ off the `$a$' branch that is distance at least $2$ from $v$.  Then we have $\y_p=0$.
\end{prop}

\begin{proof}
We consider three cases.

\textbf{Case 1.} Suppose $p<k$ and that there is another leaf $m_0<k$ on a sub-branch attached at a vertex $w$ strictly between $v$ and $B$.  Choose such an $m_0$ such that $w$ is as close to $v$ as possible.  Setting $i=k$ and applying Lemma \ref{lem:i=k} with $m_1=p$ and $m_2=m_0$, we have $$s_{p}\y_{m_0}=s_{m_0}\y_{p}.$$  Note that since $k$ is minimal among all non-$a$ branches at $v$, in the tree $T|_i$ we have that $i=k$ is adjacent to vertex $w$, and hence $s_{m_0}\neq 0$ while $s_p=0$.  Thus, the equation above simplifies to $\y_p=0$.  

\textbf{Case 2.} Suppose $p<k$ and that we are not in Case $1$.  Since $B$ is at least two steps away from $v$, there is another leaf $i>k$ on a sub-branch between $B$ and $v$; let $i$ be the smallest such element, so that it is minimal in its sub-branch.   Now we consider Lemma \ref{lem:i>k} with $i$ and $m=p$, and notice that, since $m$ is on the $a$ branch at $v_i$ in $T|_i$ we have $s_m=0$.  Furthermore, we have $t_m=t_i=0$ since $i,m$ are both on the $a$ branch at $v_n$ in $T$.  Therefore, Lemma \ref{lem:i>k} simplifies to $-s_k\y_m=0$.  Since $i$ is the minimum of its sub-branch, it separates $a$ from $k$ in $T|_i$, so $s_k\neq 0$. Therefore, the equation again simplifies to $\y_m=0$ (and since we set $m=p$ we have $\y_p=0$).

\textbf{Case 3.} Suppose $p>k$ and there is another sub-branch between $B$ and $v$ whose minimal element is greater than $k$.  Then the computation in Case 2 goes through again.

Otherwise,  we set $i=p$, and we have some $m<i$ on a sub-branch between $B$ and $v$.  We wish to show $\y_{i}=0$.  Now we again apply Lemma \ref{lem:i>k}, where in this case we have $t_{m}=t_{i}=0$, and note that $s_{m}=s_k\neq 0$ because we assumed in the Proposition statement that $i=p$ is minimal in its sub-branch.  Thus, the equation simplifies to $s_{m}\y_{i}=0,$ and we are done.
\end{proof}

\begin{prop}[Equality on sub-branches]\label{prop:equal-sub-a}
 In the tree $T$, suppose $p$ and $q$ are on the same sub-branch off the `$a$' branch.  Then we have the tangent equation $\y_p=\y_q$.
\end{prop}

\begin{proof}
In the first two cases below, we suppose both $p$ and $q$ are less than $k$; in this case we rename them to $m_1$ and $m_2$ and set $i=k$.  

  \textbf{Case 1.} Suppose $m_1,m_2$ are on a sub-branch $B$ off the $a$ branch attaching at a vertex $w$, such that there is no $m<k$ on any other sub-branch attached between $w$ and $v$.  Then since $i=k$ is minimal along the $k$ branch from $w$, in $T|_i$ the leaf $i=k$ separates $B$ from $a$.  Hence, in this case we have $s_{m_1}=s_{m_2}\neq 0$, and so the equation $$s_{m_1}\y_{m_2}=s_{m_2}\y_{m_1}$$ from Lemma \ref{lem:i=k} reduces to $\y_{m_1}=\y_{m_2}$ as desired.

  \textbf{Case 2.}  Suppose instead that there is some $m<k$ branching off the $a$ branch in between the $m_1,m_2$ branch and $v$.  Choose such an $m$ that is as close to $v$ as possible.  Then since $i=k$ attaches to the base of $m$'s sub-branch in $T|_i$, we have $s_{m}\neq 0$, and we also have $s_{m_1}=s_{m_2}=0$.  By Lemma \ref{lem:i=k} applied to $m_1$ and $m$, it follows that $\y_{m_1}=0$.  Similarly, applying Lemma \ref{lem:i=k} to $m_2$ and $m$, it follows that $\y_{m_2}=0$.  Thus, $\y_{m_1}=\y_{m_2}$ (and both equal $0$ in this case).
  
  \textbf{Case 3.}  Now suppose $p>k$ (and assume without loss of generality that $p>q$) and let $B$ be the sub-branch containing $p,q$.  We show by induction on $p$ that $\y_p=\y_q$.  For the base case, if $p$ is minimal in $B$ and $q=p$ then there is nothing to show.
  
  For the induction step, suppose we already know that $\y_{q_1}=\cdots =\y_{q_r}$ where $q_1,\ldots,q_r$ are the elements of branch $B$ less than $p$.  Set $i=p$, and in the tree $T|_i$, there must be at least one leaf $m$ from sub-branch $B$ that $i$ separates from $a$.  Let $m$ be such a leaf, and notice that $m$ is among $q_1,\ldots,q_r$.  Then we have $s_m\neq 0$, $s_k=0$, and $t_i=t_m=0$.  Thus, Lemma \ref{lem:i>k} becomes $$s_m\y_m-s_m\y_i=0,$$ and so $\y_m=\y_i$.  Since we substituted $i=p$, we see that $\y_p$ is equal to one of $\y_{q_1},\ldots,\y_{q_r}$, and so it is equal to all of them, and therefore is equal to $\y_q$.
\end{proof}

\begin{corollary}
If $p$ is any point on a sub-branch off the `$a$' branch that is distance at least $2$ from $v$, then $\y_p=0$.
\end{corollary}

\begin{proof}
  This follows by combining Propositions \ref{prop:two-away-a} and \ref{prop:equal-sub-a}. 
\end{proof}

We finally compare vertices along sub-branches that attach to the same point on the $a$ branch, which is possible since $T$ is not necessarily trivalent.

\begin{prop}[Comparison of sub-branches]\label{prop:compare-subs-a}
 Let $w$ be the vertex closest to $v$ along the path from $v$ to $a$, and suppose there are two sub-branches $B_1,B_2$ off the `$a$' branch that attach at $w$.  Let $m_1,m_2$ be the minimal elements of $B_1,B_2$.  Then there are nonzero constants $c_1$ and $c_2$ such that $c_1\y_{m_1}=c_2\y_{m_2}$.
\end{prop}

\begin{proof}
\textbf{Case 1.} First suppose $m_1,m_2<k$, and set $i=k$.  Then since $w$ is as close as possible to $v$, in the tree $T|_i$ we have that $k$ separates $m_1$ and $m_2$ from $a$, and so $s_{m_1},s_{m_2}\neq 0$.  This means that the equation from Lemma \ref{lem:i=k}, $$s_{m_1}\y_{m_2}=s_{m_2}\y_{m_1},$$ is nonvanishing with nonzero constants on each side, as desired.

\textbf{Case 2.}  Now suppose one of the two minima is greater than $k$, and rename the minimum elements $i,m$ with $i>m$.  Then in Lemma \ref{lem:i>k} we have $t_i=t_m=0$ and the equation becomes $$(s_m-s_k)\y_m=s_m\y_i.$$  Since $m$ is on a different branch at $v_i$ than $k$ in $T|_i$, we have $s_m\neq s_k$ so the left hand side coefficient is nonvanishing, and we also see that $s_m\neq 0$ since $i$ separates $m$ from $a$ in the forgotten tree.

This completes the proof.
\end{proof}

We summarize and enumerate the equations along the $a$ branch as follows.  For brevity, we make the following definition.

\begin{definition}
A sub-branch off the $a$ branch is \textbf{nonzero} if it attaches to the vertex $w$ adjacent to $v$, and it is \textbf{zero} otherwise.
\end{definition}

\begin{corollary}[Equations on the $a$ branch]\label{rmk:a-branch}
  If there are $L_0\ge 2$ leaves on the `$a$' branch (including $a$), the following $L_0-2$ equations are independent tangent equations satisfied by the $\y$ variables:
  \begin{itemize}
      \item $\y_p=0$ whenever $p$ is on a zero sub-branch of the `$a$' branch,
      \item $\y_{r}=c_r\y_{m_0}$ for some $c_r\neq 0$, where $r,m_0$ are on a nonzero sub-branch and $m_0$ is the minimal leaf on all nonzero sub-branches. 
  \end{itemize}
  These eliminate the $\y$ variables corresponding to all $L_0$ leaves on the `$a$' branch except for $a$ and $m_0$.
\end{corollary}

\begin{proof}
  Propositions \ref{prop:two-away-a} and \ref{prop:equal-sub-a} combine to provide the $\y_{p}=0$ equations.  Then, Propositions \ref{prop:equal-sub-a} and \ref{prop:compare-subs-a} provide the equations of the form $\y_{r}=c_r\y_{m_0}$. These are clearly linearly independent since they have distinct leading terms (ordering the $\y_i$'s by the value $i$).
\end{proof}
Note that if $a$ is the only element of its branch, i.e. $L_0=1$, then there are no equations obtained in this case.

\subsection{Equations along the \texorpdfstring{$k$}{k} branch}

We now analyze equations coming from comparing $i,m$ values along the $k$ branch at $v$.  Notice that every entry along the $k$ branch is larger than $k$, so we get no equations from the $i=k$ case (Lemma \ref{lem:i=k}).  Thus, Lemma \ref{lem:i>k} is the only type of equation that could apply, and notice also that $t_i=t_m=1$ for any $i,m$ along the $k$ branch.   Hence, the equation in Lemma \ref{lem:i>k} always has its right hand side equal to $0$ in this setting, and the left hand terms simplify as well; rearranging it becomes:
\begin{equation}\label{eqn:k-branch}
    s_m\y_m=(s_m-s_k)\y_i.
\end{equation}

\begin{definition}
A \textbf{sub-branch along the $k$ branch} is defined in an identical manner to a sub-branch along the $a$ branch, but for the branch at $v$ that contains $k$ (see Definition \ref{def:sub-branch}).
\end{definition}

The following diagram will come in handy throughout this subsection.

\begin{center}
    \includegraphics{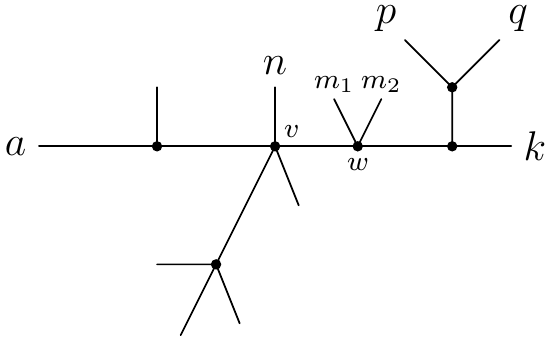}
\end{center}

\begin{prop}[Two away from $v$ on $k$ branch]\label{prop:two-away-k} 
   In $T$, suppose $q$ is the minimal element of a sub-branch $B$ off the $k$ branch that attaches at least 
   distance $2$
   away from $v$ on the path from $v$ to $k$. 
   Then on the tangent space, we have $\y_q=0$.
\end{prop}

\begin{proof}
Let $B'$ be a sub-branch that attaches at a point strictly closer to $v$ than $B$ along the path from $k$ to $v$. 
Let $r$ be the minimal element of $B'$.  

\textbf{Case 1.} Suppose $r>q$.  Then we set $i=r$ and $m=q$, and we find that $s_m=s_k=1$, so Equation \eqref{eqn:k-branch} becomes $\y_m=0$ as desired.

\textbf{Case 2.}  Suppose $r<q$.  Then we set $i=q$ and $m=r$, and since $i$ is minimal in its sub-branch by the assumption in the Proposition statement, we have $s_m=0$ and $s_k=1$.  Thus, Equation \eqref{eqn:k-branch} becomes $\y_i=0$ as desired.
\end{proof}

\begin{prop}\label{prop:equal-sub-k}
  (Equality on sub-branches off $k$ branch.)  In the tree $T$, suppose $p$ and $q$ are in the same sub-branch off the $k$ branch.  Then we have the tangent equation $\y_p=\y_q$.
\end{prop}

\begin{proof}
	Notice that since we are on the $k$ branch, we have $p,q>k$.  Also assume $p>q$, and set $i=p$, $m=q$. If $i$ and $m$ are the two smallest elements of their sub-branch $B$, then we have $s_m=1$ and $s_k=0$, so Equation \eqref{eqn:k-branch} becomes $\y_i=\y_m$.
	
	Now, we assume for strong induction that we have $\y_{m_1}=\y_{m_2}=\cdots=\y_{m_r}$ for the first $r$ smallest elements $m_1<m_2<\cdots<m_r$ of $B$, and show the claim holds for $i$ being the smallest element of $B$ greater than $m_r$.   In this case, we have that in $T|_i$, there is some element $m_t$ among $m_1,\ldots,m_r$ that $i$ separates from $a$, so for this value $m_t$ we have $s_{m_t}\neq 0$, and $s_k=0$, so Equation \eqref{eqn:k-branch} becomes $\y_i=\y_{m_t}$.  Thus, $\y_i$ is also equal to all the lower $\y_m$'s, and the induction is complete. 
\end{proof}

\begin{prop}\label{prop:compare-subs-k}
	Let $w$ be the vertex closest to $v$ along the path from $v$ to $k$, and suppose there are two sub-branches $B_1,B_2$ off the $k$ branch that attach at $w$.  Let $m_1,m_2$ be the minimal elements of $B_1,B_2$.  Then there are nonzero constants $c_1$ and $c_2$ such that $c_1\y_{m_1}=c_2\y_{m_2}$.
\end{prop}

\begin{proof}
    We assume without loss of generality that $m_1>m_2$, and note that both are greater than $k$.  Setting $i=m_1$ and $m=m_2$, we have that $s_m\neq 0$ and $s_m\neq s_k$, so Equation \eqref{eqn:k-branch} has nonvanishing constants on both sides, as desired.
\end{proof}

Just as in the $a$ branch, Propositions \ref{prop:two-away-k}, \ref{prop:equal-sub-k}, and \ref{prop:compare-subs-k} give us $L_k-2$ independent equations all together, where there are $L_k\ge 2$ leaves on the $k$ branch (including $k$ itself).
	
If instead $L_k=1$, that is, if $k$ is alone on its branch, then there are $0$ equations from this setting.  We summarize this result as follows, using the terms \textbf{nonzero} and \textbf{zero} respectively to refer to sub-branches of the $k$ branch that are distance one away from $v$ or more than one, respectively.

\begin{corollary}[Equations on the $k$ branch]\label{rmk:k-branch}
  If there are $L_k\ge 2$ leaves on the $k$ branch (including $k$), the following $L_k-2$ equations are independent tangent equations satisfied by the $\y$ variables:
  \begin{itemize}
      \item $\y_q=0$ whenever $q$ is on a zero sub-branch of the $k$ branch,
      \item $\y_{r}=c_r\y_{m_1}$ for some $c_r\neq 0$, where $r,m_1$ are on a nonzero sub-branch  and $m_1$ is the minimal leaf on all nonzero sub-branches. 
  \end{itemize}
  These eliminate the $\y$ variables corresponding to all $L_k$ leaves on the $k$ branch except for $k$ and $m_1$.
\end{corollary}

\subsection{Equations within a \texorpdfstring{$t$}{t} branch}

We now consider the equations that arise from setting $i,m$ to be leaves on the same branch at $v$, besides the `$a$' and `$k$' branches.  We will call these branches `$t$'\textbf{ branches} since their $y$ coordinates are nonzero values of $t$.

\begin{center}
    \includegraphics{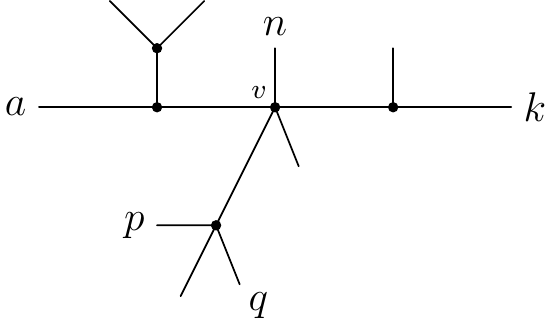}
\end{center}

We claim the following.

\begin{prop}[Equality mod $x$ on a $t$ branch]\label{prop:t-branch}
    Let $p,q$ be leaves on a non-$a$, non-$k$ branch $B_j$ at $v$.  Then there is an equation on the tangent space of the form $$\y_p=\y_q+\sum_{i<n} c_i\x_k^{(i)}$$  for some (possibly zero) constants $c_i$.
\end{prop}

\begin{proof}
    We show this by induction on $\max(p,q)$.  For the base case, if $p=q$ is the minimum element on the branch then we automatically have $\y_p=\y_q$. 
    
    Now, suppose there exist such equations among all pairs $p,q$ on the branch less than $i_0$.  We show that there is also an equation of this form for $(p,q)=(i_0,m)$ for any $m<i_0$ on the branch.  Note that by the inductive hypothesis, it suffices to show this for any particular $m<i_0$; we choose an $m$ that $i_0$ separates from $v$ in $T|_{i_0}$, which exists because $T|_{i_0}$ is at least trivalent.  
    
    Then note that for $i=i_0$ we have $s_m\neq 0$, $s_k=0$, and $t_i=t_m=:t$ for some $t\neq 0,1$.  Thus, Lemma \ref{lem:i>k} becomes $$s_m(\y_m-\y_{i_0})=t(1-t)\x_k^{(i_0)}.$$  Since $s_m\neq 0$, this is easily rearranged into the desired form.
\end{proof}

We can enumerate the new independent equations we obtain from Proposition \ref{prop:t-branch} as follows.

\begin{corollary}[Equations on a $t$ branch] \label{rmk:other-branches}
  In a given `$t$' branch $B_t$, let $m_t$ be its minimal leaf, and suppose it has $L_t$ leaves total. Then the equations
  $$\y_p=\y_{m_t}+\sum_{i<n} c_i x_k^{(i)}$$ for each $p\neq m_t$ on $B_t$ are $L_t-1$ independent equations on the tangent space.
\end{corollary}

\begin{proof}
  There is one such equation for each $p\neq m_j$, giving $L_t-1$ equations.  These are independent because they have distinct leading terms $\y_p$.
\end{proof}

\subsection{Comparing the branches}

Finally, suppose the number of `$t$' branches is $r\ge 1$.  We now find $r+1$ more independent equations (or $r$ or $r-1$ if one or both of $L_0,L_k$ is $1$, respectively) that are also independent of those found above, by comparing the minimal elements of each $t$ branch with each other, as well as with the nonzero sub-branches of the $a$ and $k$ branches.

In particular, let $m_1,\ldots,m_r$ be the minimal elements of the `$t$' branches, say in order $m_1<m_2<\cdots<m_r$.  Set $i=\max\{m_1,\ldots,m_r\}=m_r$.  

\begin{center}
    \includegraphics{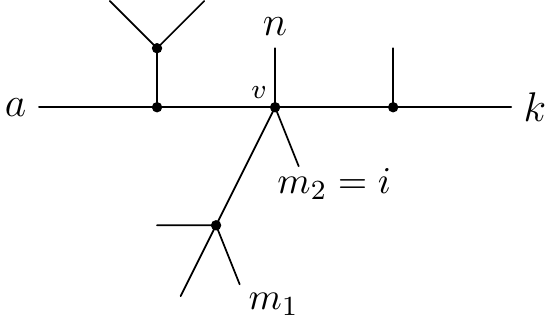}
\end{center}

Then we first have the following, which gives us $r-1$ new equations:

\begin{prop}[Comparing two $t$ branches] \label{prop:t-vs-t}
  For each $m=m_j$ for $1\le j<r$, we have an equation in the tangent space of the form $$c_1\y_i+c_2\y_{m}=d_1\x_k+d_2\x_{m}$$ for some nonzero constants $c_1,c_2,d_1,d_2$, where $i=m_r=\max\{m_1,\ldots,m_r\}$.
\end{prop}

\begin{proof}
    Notice that we can set $s_k=1$ in Lemma \ref{lem:i>k}, so in this case it simplifies to 
    $$(s_m+t_i-1)\y_m+(t_m-s_m)\y_i=t_m(1-t_i)\x_k+(t_i-t_m)\x_m.$$  To see that all of these constants are nonzero, on the right hand side we clearly have $t_i\neq 1$, $t_m\neq 0$, and $t_i\neq t_m$ by construction.  
    
    For the left hand side, consider $t_m-s_m$; on the $n$'s component of the curve, if we coordinatize so that the $a$ branch is at $0$, the $k$ branch is $1$, and $n$ is at $\infty$, then $t_m$ is the coordinate of the $m$ branch.  Instead, if we recoordinatize so that the $i$ branch is at $\infty$ rather than $n$ (but $0$ and $1$ stay the same) then $s_m$ is the coordinate of the $m$ branch. Since these are two different coordinatizations of the same point we have $t_m\neq s_m$ so $t_m-s_m\neq 0$.
    
    Now consider the coefficient of $\y_m$, which we can write as $s_m-(1-t_i)$.  Notice that the projective transformation that sends $t_i$ to $\infty$ (but also sends $0\to 0$ and $1\to 1$) is $$z\mapsto \frac{z(1-t_i)}{z-t_i}=\frac{1-t_i}{1-(t_i/z)}.$$  We see that this map then sends $\infty$ to $1-t_i$, so $1-t_i$ is the coordinate of the marked point $n$ from the perspective of the $i$ branch on the $n$ component.  This is at a different point from $m$ from the perspective of $i$, so $s_m\neq 1-t_i$, and therefore $s_m-(1-t_i)$ is nonzero as desired.
 \end{proof}

Next, suppose $L_0>1$.  Then we have an additional equation between $m_r$ and the minimum element $m_0$ of all the nonzero sub-branches of the $a$ branch.  The following diagram illustrates the next two propositions.

\begin{center}
    \includegraphics{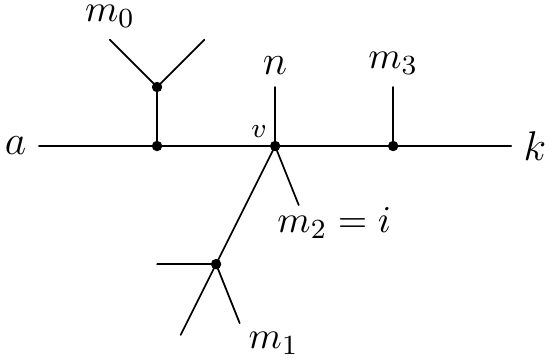}
\end{center}

\begin{prop}[Comparing a $t$ branch to the $a$ branch]\label{prop:t-vs-a}
  With $m_0$ defined as above, we have a tangent equation of the form $\y_{m_0}=c_1\x_k+c_2\x_{m_0}+c_3\x_{m_r}$ for some (possibly zero) constants $c_1,c_2,c_3$.
\end{prop}

\begin{proof}
  We consider two cases.
  
  \textbf{Case 1.} Suppose $m_0<m_r$.  Then we set $i=m_r$ and $m=m_0$, and in this case we have $t_m=s_m=0$.  We also have $s_k=1$, since $k$ is still the smallest element in the non-$a$ branches at $v_i$ in $T|_i$ since $i$ is minimal in its branch.  Setting $t=t_i$, Lemma \ref{lem:i>k} simplifies to $$(t-1)\y_m=t(1-t)\x_k+t\x_m,$$ and note that $t\neq 1$ so we can divide both sides by $t-1$ to get the result. 
  
  \textbf{Case 2.} Suppose $m_0>m_r$.  Then we set $i=m_0$ and $m=m_r$, and in this case we have $t_i=0$, and again $s_k=1$ since $i$ is the minimal element of the nonzero sub-branches of the $a$ branch, and therefore in $T|_i$, the leaf $k$ is the smallest element not on the $a$ branch from $v_i$.  We similarly have $s_m=1$ since $m$ is in the same branch as $k$ from $v_i$ in $T|_i$.  Thus, Lemma \ref{lem:i>k} simplifies to $$(t_m-1)\y_i=t_m\x_k-t_m\x_m.$$ Finally note that $t_m\neq 1$ since $m$ is on a different branch than $k$, and so we can divide both sides by $t_m-1$.
\end{proof}

Finally, suppose $L_k>1$.  Then we have an additional equation between $m_r$ and the minimal element $m_{r+1}$ of the nonzero sub-branches on the $k$ branch:

\begin{prop}[Comparing a $t$ branch to the $k$ branch]\label{prop:t-vs-k}
  With $m_{r+1}$ defined as above, we have a tangent equation of the form $\y_{m_{r+1}}=c_1\x_k+c_2\x_{m_{r+1}}+c_3\x_{m_r}$ for some (possibly zero) constants $c_1,c_2,c_3$.
\end{prop}

\begin{proof}
  We consider two cases.
  
  \textbf{Case 1.} Suppose $m_{r+1}<m_r$.  Then we set $i=m_r$ and $m=m_{r+1}$, and in this case we have $t_m=s_m=s_k=1$.  Setting $t=t_i$, Lemma \ref{lem:i>k} simplifies to $$t\y_m=(1-t)\x_k+t\x_m,$$ and note that $t\neq 0$ so we can divide both sides by $t$ to get the desired equation.
  
  \textbf{Case 2.} Suppose $m_{r+1}>m_r$.  Then we set $i=m_{r+1}$ and $m=m_r$, and in this case we have $t_i=1$, $s_k=1$, and $s_m=0$ since $m$ is now on the same branch as $a$ at $v_i$ in $T_i$.  Thus, Lemma \ref{lem:i>k} simplifies to 
  $$t_m\y_i=t_m\x_k+(1-t_m)\x_m.$$ Finally note that $t_m\neq 0$ since it is on a different branch than $a$, and so we can divide both sides by $t_m$ to obtain the result in this case (since $i=m_{r+1}$). 
\end{proof}

Putting together Propositions \ref{prop:t-vs-t}, \ref{prop:t-vs-a}, and \ref{prop:t-vs-k}, we have the following.

\begin{corollary}[Comparison equations with $t$ branches] \label{rmk:comparison-t}
  If $r\ge 1$ and we have $L_0,L_k\ge 2$, let $m_{1}<\cdots<m_r$ be the minimal elements of the $t$ branches, and let $m_0, m_{r+1}$ be the minimal elements of the nonzero sub-branches on the $a$ and $k$ branch respectively.   We have the following $r+1$ independent tangent equations:
  \begin{itemize}
      \item $\y_{m_j}=c_jy_{m_r}+\sum_{i,p<n}d_p^{(i)}\x_p^{(i)}$, for $j = 1, \ldots, r$,
      \item $\y_{m_0}=\sum_{i,p<n}e_p^{(i)}\x_p^{(i)}$
      \item $\y_{m_{r+1}}=\sum_{i,p<n}f_p^{(i)}\x_p^{(i)}$
  \end{itemize}
  for appropriate constants $c,d,e,f$, where the $c_j$ constants are nonzero.  Moreover, these equations are independent of the equations stated in Corollaries \ref{rmk:a-branch}, \ref{rmk:k-branch}, and \ref{rmk:other-branches}.
  
  When $L_0=1$ or $L_k=1$ (but not both) then either the second or third equation disappears and there are only $r$ independent equations, and when both are $1$ there are $r-1$. 
\end{corollary}

\begin{proof}
  These equations are again independent since they each express a distinct $y$ variable (that are precisely the ones not eliminated by Corollaries \ref{rmk:a-branch}, \ref{rmk:k-branch}, and \ref{rmk:other-branches}) in terms of $y_{m_r}$ and the lower $x$ variables.  There are $r-1$ equations of the first type and one each of the second and third.
\end{proof}

We can now put together the corollaries at the end of each subsection above to show that we have enough equations if $r>1$.

\begin{corollary}\label{cor:summary}
  If $r\ge 1$, the following give $n-1$ independent tangent equations involving the $\y$ variables:
  \begin{enumerate}
      \item $\y_{u}=c_u\y_{m_{0}}$ for some $c_u$ (possibly $0$), for each $u\neq a,m_0$ on the $a$ branch,
      \item $\y_{u}=c_u\y_{m_{r+1}}$ for some $c_u$ (possibly $0$), for each $u\neq k,m_{r+1}$ on the $k$ branch,
      \item $\y_u=\y_{m_j}+(\x\text{ terms})$ for $u\neq m_j$ on the $t$ branch with minimum $m_j$,
      \item $\y_{m_j}=c_j\y_{m_r}+(\x\text{ terms})$ for each $j\neq r$ in $\{0,1,2,\ldots,r+1\}$.
  \end{enumerate}
  If $r=1$ and either $L_0=1$ or $L_k=1$, then there are also a total of $n-1$ independent equations among the above cases.
\end{corollary}

\begin{proof} 
  For $r\ge 1$, the equations (1) summarize Corollary \ref{rmk:a-branch} (combining the two possibilities $c_u = 0$, $c_u \ne 0$). The second, third, and fourth list items similarly summarize the equations arising in Corollaries \ref{rmk:k-branch}, \ref{rmk:other-branches}, and \ref{rmk:comparison-t}. 
  
  Taken together, these equations eliminate all $\y_q$ variables for $q$ being every leaf except $a,k,n,m_r$.  This gives $n-1$ equations since there are a total of $n+3$ leaves.
  
  If $r=1$, there are no $t$ branches, so the equations (3) and (4) do not apply. If $L_0 = 1$, list item (2) directly gives $n-1$ independent equations. Similarly, if $L_k = 1$, list item (1) gives $n-1$ independent equations.
\end{proof}

\subsection{Final case: \texorpdfstring{$n$}{n} on a trivalent component with two nodes}

In the previous section, we assumed $r>1$ (or $r=1$ and either $L_0$ or $L_k$ is $1$), meaning there was at least one branch at $v=v_n$ besides the $a$ and $k$ branches.  However, in the case that the vertex $v$ is trivalent and $L_0,L_k\ge 2$, Corollary \ref{cor:summary} does not apply. Indeed, in this case we have only $(L_0-2)+(L_k-2)=(n+3)-5=n-2$ tangent equations from the $a$ branch and $k$ branch. Assuming $r=1$ and $L_0,L_k>1$, we now find one final equation. As discussed in Remark \ref{rmk:why-exceptional}, this last equation must involve only the coordinates $x_j^{(i)}$ where $i < n$.

Let $m_0,m_1$ be the minima of the nonzero sub-branches on the $a$ and $k$ branch, respectively, as shown below.

\begin{center}
    \includegraphics{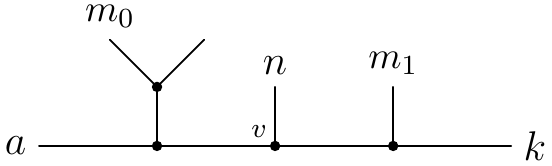}
\end{center}

\begin{prop}\label{prop:extra}
  Assume $r=1$ and $L_0, L_k > 1$.
  Let $m=\min(m_0,m_1)$ and $i=\max(m_0,m_1)$. Then $$\x_m^{(i)}=0.$$
\end{prop}

\begin{proof}
  \textbf{Case 1.}  Suppose $m_0<m_1$, so that $m=m_0$ and $i=m_1$.  Then we have $t_m=0$, $t_i=1$, $s_m=0$, $s_k=1$ so Lemma \ref{lem:i>k} becomes $\x_m=0$ as desired.
  
  \textbf{Case 2.}  Suppose $m_0>m_1$, so that $m=m_1$ and $i=m_0$.  Note that $m>k$ since $k$ is minimal in the $k$ branch, so $i>k$ as well. Then we have $t_i=0$, $t_m=1$, $s_m=1$, $s_k=1$ so Lemma \ref{lem:i>k} becomes $\x_m=\x_k$.  But since $k$ is also the chart at step $i$ (since $i$ is the minimal leaf on all nonzero sub-branches of the $a$ branch), we also have $\x_k=0$ in this setting.  Thus, $\x_m=0$ is again a tangent equation.
\end{proof}

\begin{example}\label{ex:example}
  We can now combine Proposition \ref{prop:extra} with the equations from Corollary \ref{cor:summary} to write out a full example of the equations defining the tangent space for a fixed curve.  Let $C$ be a curve whose dual tree is shown at left in Figure \ref{fig:example}, and set $t=s^{(4)}_c$ to be the parameter of the unique `$t$' branch in $\pi_5(C)$ (shown at right in Figure \ref{fig:example}).

\begin{figure}[b]
    \centering
    \includegraphics{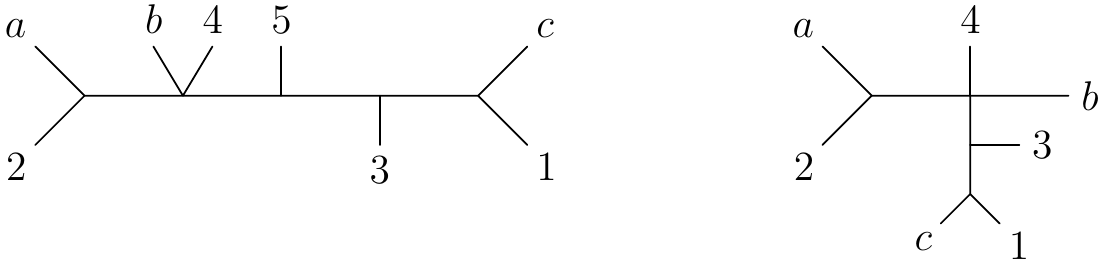}
    \caption{At left, the dual tree of the curve $C$ used in Example \ref{ex:example}.  At right, the dual tree of $\pi_5(C)$. }
    \label{fig:example}
\end{figure}
  
  Set the coordinates $\uu\times \vv\times\ww\times \x\times\y$ for the local tangent variables at $\Omega_5(C)$ in $\PP^1\times \PP^2\times \PP^3\times \PP^4\times \PP^5$ for simplicity of notation (avoiding superscripts). Then we have the following ten independent tangent equations:
  \begin{multicols}{3}
  \begin{itemize}
      \item $\y_2=0$ 
      \item $\y_1=0$
      \item $\y_4=(1-t)\y_b$
      \item $\ww_b=0$
      \item $\x_1=\x_c+t(1-t)\uu_b$
      \item $\x_3=\x_c+t(1-t)\ww_b$
      \item $\x_2=\frac{t}{1-t}(\vv_b-\vv_c)$
      \item[\vspace{\fill}]
      \item $\ww_2=0$ 
      \item $\uu_b=0$
      \item $\vv_1=\vv_c$
      \item[\vspace{\fill}]
  \end{itemize}
  \end{multicols}
  The four equations in the left hand column of the list above are the $5-1=4$ equations we get from Corollary \ref{cor:summary} and Proposition \ref{prop:extra}.  The next three are the equations we get from analyzing $\pi_5(C)$ using Corollary \ref{cor:summary}, since it is a non-exceptional curve.  The final three are the result of analyzing $\pi_4(\pi_5(C))$ and $\pi_3(\pi_4(\pi_5(C)))$ in a similar manner.
\end{example}

We finally show that the equation obtained in Proposition \ref{prop:extra} is indeed new, in the sense that it does not hold in the corresponding tangent space to $\Mbar_{0,S\setminus n}$.

\begin{lemma}\label{lem:nonzero-coord}
Let $C$ be a point in $\overline{M}_{0,S}$ and $T$ its dual tree. Suppose $n$ is on a component of $C$ with exactly three special points, and let $i$ and $m$ be as in Proposition~\ref{prop:extra}. Letting $C' = \pi_n(C)\in \overline{M}_{0,S\setminus n}$, the coordinate $\underline{x}_m^{(i)}$ is not identically zero on $T_{C'}\overline{M}_{0,S\setminus n}$.
\end{lemma}

\begin{figure}[b]

  \centering
  \includegraphics[scale=0.95]{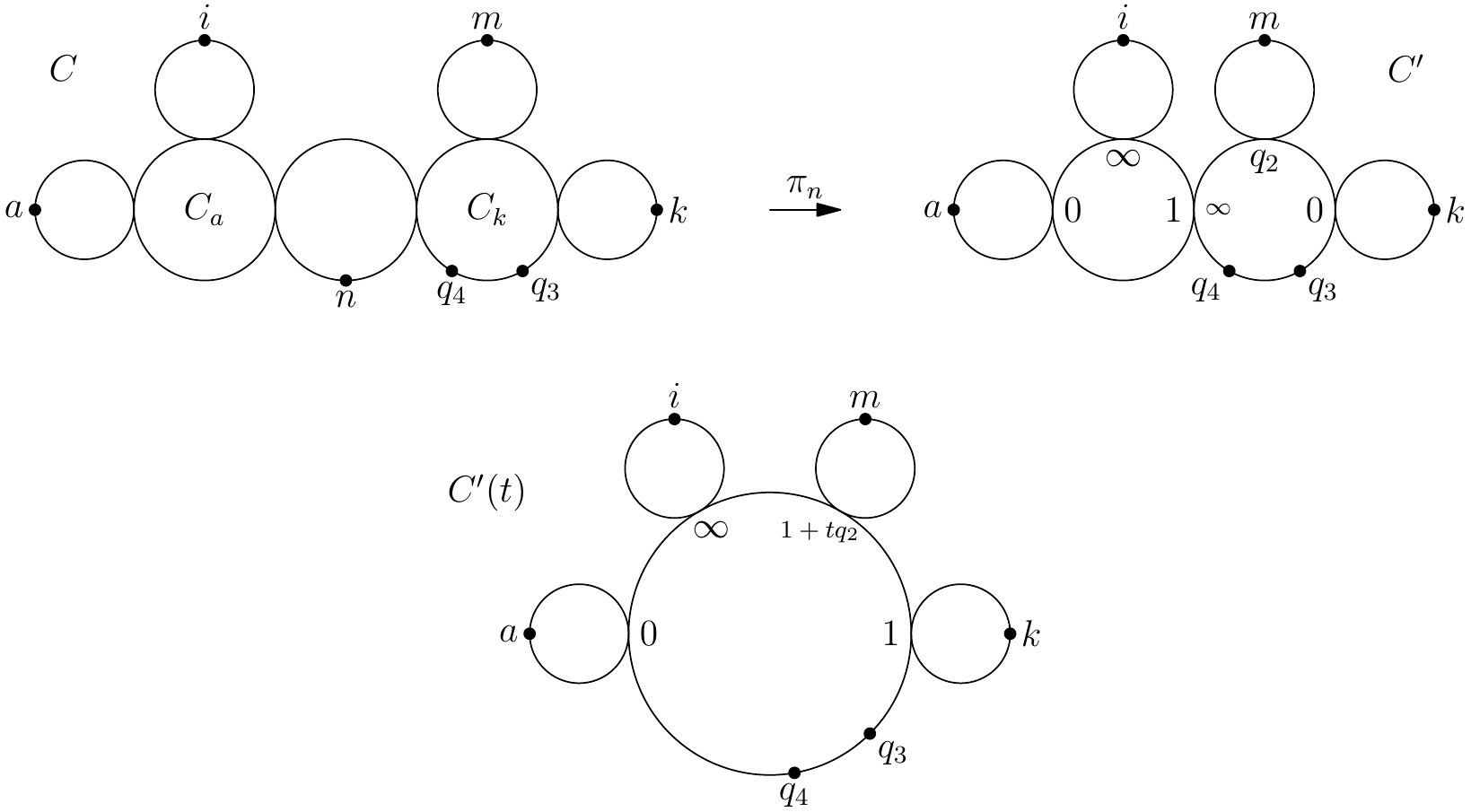}
  
\caption{Constructing a one-parameter path $C'(t)$ in $\Mbar_{0,S\setminus n}$ passing through $C'=C'(0)$ whose derivative is nonzero at $C'$ in the variable $\x^{(i)}_m$.}
    \label{fig:sean}
\end{figure}

\begin{proof}
  Let $v_a$ be the vertex of $T$ adjacent to $v$ on the $a$ branch and $v_k$ be the one on the $k$ branch. Let $C_a$ and $C_k$ be the components of $C'$ corresponding to these vertices, respectively, as shown in Figure \ref{fig:sean}.

  For a parameter $t$, we will define a curve $C'(t)\in \overline{M}_{0,S\setminus n}$ such that $C'(0)=C'$ and the tangent vector to $C'(t)$ at $t=0$ has $\x_m^{(i)}\neq 0$.  We consider the cases $m_0>m_1$ and $m_0<m_1$ separately.  
  
  \textbf{Case 1.}  Suppose $i = m_0$ and $m = m_1$. Choose coordinates on $C_a$ so that the special point closest to $i$ is at $\infty$, the node $C_a\cap C_k$ in $C'$ is at coordinate $1$, and the special point closest to $a$ is at coordinate $0$. Let $q_0,q_1,\dots, q_\ell$ be the special points on $C_k$ where $q_0$ is the node $C_a\cap C_k$, $q_1$ is the special point closest to $k$, and $q_2$ is the special point closest to $m$, and choose coordinates on $C_k$ so that $q_0 = \infty$ and $q_1 = 0$ (abusing notation by identifying $q_j$ with its coordinate on $C_k$). 
  
  Then we define $C'(t)$ to be the point of $\overline{M}_{0,S\setminus n}$ obtained by smoothing the node $C_a\cap C_k$ so that all special points on $C_a$ remain at the same coordinates and each $q_j$ is at coordinate $1+tq_j$. Then, by construction, we have that $C'(0) = C'$, and after setting $x^{(i)}_k = 1$ we have $x^{(i)}_m = 1+tq_2$.   Thus, the tangent vector to the curve $C'(t)$ at $t=0$ has $\underline{x}_m^{(i)}$ coordinate $q_2$, which is nonzero.
  
  \textbf{Case 2.} Suppose $i=m_1$ and $m = m_0$. Change coordinates on $C_k$ so that the special point closest to $a$ is at $0$, the special point closest to $i$ is at $\infty$, and the special point closest to $k$ is at $1$. Let $q_0,q_1,\dots, q_\ell$ be the special points on $C_a$ where $q_0$ is the the special point closest to $a$, $q_1$ is the node $C_a\cap C_k$, and $q_2$ is the special point closest to $m$. Change coordinates on $C_a$ so that $q_0 = 0$ and $q_1=\infty$. 
  
  Then we define $C'(t)$ to be the point of $\overline{M}_{0,S\setminus n}$ which is the result of ``smoothing'' the node $C_a\cap C_k$ so that all special points on $C_k$ remain at the same coordinates, and each $q_j$ is at coordinate $tq_j$. Then, by construction, we have $C'(0) = C'$, and after setting $x_k^{(i)} = 1$ we have $x_m^{(i)} = tq_2$. Thus, the tangent vector to the curve $C'(t)$ at $t=0$ has $\underline{x}_m^{(i)}$ coordinate $q_2$, which is nonzero.
\end{proof}

We can finally prove the main result.

\begin{MainThm}
  We have $\Omega_n(\Mbar_{0,S}) = \MR_n$ as subschemes of $\PP^1 \times \cdots \times \PP^n$.
\end{MainThm}

\begin{proof}
  We proceed by induction on $n$.  For the base case, $n=1$, the space $\Mbar_{0,\{a,b,c,1\}}$ is just $\PP^1$, the embedding $\Omega_1$ is an isomorphism, and there are no Monin--Rana equations.
  
  For the induction step, suppose the claim holds for $n-1$. Let $\vec{x} \in \MR_n \subseteq \PP^1\times \cdots \times \PP^n$, let $C$ be its associated curve and $T$ its dual tree.
  
  In the cases where Corollary \ref{cor:summary} applies, we obtain $n-1$ tangent equations with distinct leading terms in the $\PP^n$ variables. These equations are independent of the $\binom{n-1}{2}$ equations that we have by induction, since the latter only involve variables from $\mathbb{P}^1 \times \cdots \times \mathbb{P}^{n-1}$.

  In the last case, the marked point $n$ is on a component of $C$ of valency $3$, and $L_0, L_k > 1$. Then there are at least $n-2$ independent equations involving the $\PP^n$ variables, and by Lemma \ref{lem:nonzero-coord}, the equation $\x_m^{(i)}=0$ found in Proposition~\ref{prop:extra} is independent of them and of the previous $\binom{n-1}{2}$ equations.  This completes the proof. 
\end{proof}

\bibliography{myrefs}
\bibliographystyle{plain}

\end{document}